\documentclass[10pt, leqno]{article}

\usepackage{amsmath,amssymb,amsthm, epsfig}

\usepackage{hyperref}

\usepackage{amsmath}

\usepackage[pdftex]{color}

\usepackage{color}

\usepackage{ulem}

\usepackage{mathrsfs}

\usepackage{wasysym}

\usepackage{stackrel}

\title{\large{\bf Sharp Hessian estimates for fully nonlinear elliptic equations under relaxed convexity assumptions, oblique boundary conditions and applications}}
\author{\it by \smallskip \\
 Junior da S. Bessa\footnote{\noindent Universidade Federal do Cear\'{a}. Department of Mathematics. Fortaleza - CE, Brazil. \noindent \texttt{E-mail address: \url{junior.bessa@alu.ufc.br}}},\,\,\,Jo\~{a}o Vitor  da Silva\footnote{\noindent Universidade Estadual de Campinas - UNICAMP. Department of Mathematics. Campinas - SP, Brazil. \noindent \texttt{E-mail address: \url{jdasilva@unicamp.br}}}, \,\,\,Maria N.B. Frederico\footnote{\noindent Universidade Federal do Cear\'{a}. Department of Mathematics. Fortaleza - CE, Brazil. \noindent \texttt{E-mail address: \url{nildebarreto@gmail.com}}}\\ $\&$\\ Gleydson C. Ricarte \footnote{\noindent Universidade Federal Cear\'{a}. Department of Mathematics. Fortaleza, CE-Brazil 60455-760. \noindent \texttt{E-mail address: \url{ricarte@mat.ufc.br}}}
}

\newlength{\hchng}
\newlength{\vchng}
\def \R {\mathbb{R}}
\def \supp {\mathrm{supp } }

\def \Leb {\mathcal{L}^n}

\newcommand{\defeq}{\mathrel{\mathop:}=}

\newtheorem{theorem}{Theorem}[section]

\newtheorem{lemma}[theorem]{Lemma}

\newtheorem{proposition}[theorem]{Proposition}

\newtheorem{corollary}[theorem]{Corollary}

\theoremstyle{definition}

\newtheorem{definition}[theorem]{Definition}

\theoremstyle{remark}

\newtheorem{remark}[theorem]{Remark}

\numberwithin{equation}{section}

\newcommand{\intav}[1]{\mathchoice {\mathop{\vrule width 6pt height 3 pt depth  -2.5pt
\kern -8pt \intop}\nolimits_{\kern -6pt#1}} {\mathop{\vrule width
5pt height 3  pt depth -2.6pt \kern -6pt \intop}\nolimits_{#1}}
{\mathop{\vrule width 5pt height 3 pt depth -2.6pt \kern -6pt
\intop}\nolimits_{#1}} {\mathop{\vrule width 5pt height 3 pt depth
-2.6pt \kern -6pt \intop}\nolimits_{#1}}}

\begin{document}
\maketitle

\begin{abstract}

We derive global $W^{2,p}$ estimates (with $n\le p <\infty$) for viscosity solutions to fully nonlinear elliptic equations under relaxed structural assumptions on the governing operator that are weaker than convexity and oblique boundary conditions as follows:
$$
\left\{
\begin{array}{rclcl}
 F(D^2u,Du,u,x) &=& f(x)& \mbox{in} &   \Omega \\
 \beta(x) \cdot Du(x) + \gamma(x) u(x)&=& g(x) &\mbox{on}& \partial \Omega,
\end{array}
\right.
$$
for $f \in L^p(\Omega)$ and under appropriate assumptions on the data $\beta, \gamma$, $g$ and $\Omega \subset \R^n$. Our approach makes use of geometric tangential methods, which consist of importing ``fine regularity estimates'' from a limiting profile, i.e., the \textit{Recession} operator, associated with the original second-order one via compactness and stability procedures. As a result, we pay special attention to the borderline scenario, i.e., $f \in \text{BMO}_p \varsupsetneq L^{\infty}$. In such a setting, we prove that solutions enjoy $\text{BMO}_p$ type estimates for their second derivatives. Finally, as another application of our findings, we obtain Hessian estimates to obstacle-type problems under oblique boundary conditions and no convexity assumptions, which may have their own mathematical interest. A density result for a suitable class of viscosity solutions will also be addressed.

\vspace{0.2cm}

\end{abstract}

\newpage

\section{Introduction}

\hspace{0.4cm}In this manuscript, we investigate global $W^{2,p}$ estimates (with $n \le p< \infty$) for viscosity solutions to fully nonlinear elliptic equations of the form:
\begin{equation}\label{E1}
\left\{
\begin{array}{rclcl}
 F(D^2u,Du,u,x) &=& f(x)& \mbox{in} &   \Omega \\
 \mathcal{B}(x,u,Du)&=& g(x) &\mbox{on}& \partial \Omega,
\end{array}
\right.
\end{equation}
where
$$
\displaystyle \mathcal{B}(x,u,Du) = \beta(x) \cdot Du(x) + \gamma(x) u(x)
$$
is the oblique boundary condition, $\Omega \subset \mathbb{R}^n$ a bounded  domain with a regular boundary and data $f$, $g$, $\beta$ and $\gamma$ under appropriate conditions to be clarified soon.

Throughout this work, it is assumed that $F:\textit{Sym}(n)\times \R^n \times \R \times \Omega \to \R$ is a second-order operator with uniformly elliptic structure, i.e., there exist \textit{ellipticity constants} $0<\lambda \le \Lambda< \infty$ such that
\begin{equation}\label{Unif.Ellip.}
  \lambda\|\mathrm{Y}\|\leq F(\mathrm{X}+\mathrm{Y}, \varsigma, s, x)-F(\mathrm{X}, \varsigma, s, x) \leq \Lambda \|\mathrm{Y}\|
\end{equation}
 for every $\mathrm{X}, \mathrm{Y} \in \textit{Sym}(n)$ with $\mathrm{Y} \ge 0$ (in the usual partial ordering on symmetric matrices) and $(\varsigma, s, x)\in \mathbb{R}^n \times \R \times \Omega $. Moreover, our approach does not impose any extra regularity assumption on nonlinearity $F$, \textit{e.g.} concavity, convexity, or smoothness hypothesis (cf. \cite{Caff1}, \cite{CC}, \cite{Lie-Tru} and \cite{Sav07} for classical and modern references). As a matter of fact, we obtain global $W^{2,p}$-estimates, i.e.,
  $$
  \|u\|_{W^{2, p}(\Omega)} \le \mathrm{C}(\verb"universal")\left(\|u\|_{L^{\infty}(\Omega)}+ \|f\|_{L^p(\Omega)}+\Vert g\Vert_{C^{1,\alpha}(\partial \Omega)}\right)
  $$
  (with $n \le p < \infty$ and some $\alpha \in (0, 1)$) for viscosity solutions to \eqref{E1} under relaxed convexity assumptions on governing operator (cf. \cite{Kry13}, \cite{PT} and \cite{ST} for modern results) -- a kind of asymptotic convexity condition at infinity on $F$.

 To do this concept more precise, we introduce the notion of \textit{Recession operator}, whose terminology comes from Giga-Sato's work \cite{GS01} on theory of Hamilton-Jacobi PDEs.

\begin{definition}[{\bf Recession profile}]\label{DefAC}
We say that $F(\mathrm{X},\varsigma, s,x)$ is an asymptotically fully nonlinear elliptic operator if there exists a uniformly elliptic operator $F^{\sharp}(\mathrm{X}, \varsigma, s ,x)$, designated the \textit{Recession} operator, such that
\begin{equation}\tag{{\bf \text{\small{Reces.}}}}\label{Reces}
  \displaystyle F^{\sharp}(\mathrm{X}, \varsigma, s, x) \defeq  \lim_{\tau \to 0^{+}} \tau \cdot F\left(\frac{1}{\tau}\mathrm{X}, \varsigma, s, x\right)
\end{equation}
 for all $\mathrm{X} \in \textrm{Sym}(n)$, $\varsigma \in \mathbb{R}^n$, $s \in \mathbb{R}$ and $x \in \Omega$.
\end{definition}

For instance, by way of illustration, such a limiting profile \eqref{Reces} appears naturally in singularly
perturbed free boundary problems ruled by fully nonlinear equations, whose Hessian of solutions
blows-up through the phase transition of the model, i.e., $\partial\{u^{\varepsilon}> \varepsilon\}$, where $u^{\varepsilon}$ satisfies in the viscosity sense
$$
F(D^2 u^{\varepsilon}, x) = \mathcal{Q}_0(x)\frac{1}{\varepsilon} \zeta \left(\frac{u^{\varepsilon}}{\varepsilon}\right).
$$
For such approximations, we have $0< \mathcal{Q}_0 \in C^0(\overline{\Omega})$, $0\leq \zeta \in C^{\infty}(\R)$  with $\supp ~\zeta = [0,1]$. For this reason, in the above model, the limiting free boundary condition is governed by the F$^{\sharp}$ rather than F, i.e.,
$$
F^{\sharp}(z_0, \nabla u(z_0) \otimes \nabla u(z_0)) = 2\mathrm{T},  \quad z_0 \in \partial\{u_0>0\}
$$
in some appropriate viscosity sense, for a certain total mass $\mathrm{T}>0$ (cf. \cite[Section 6]{RT} for some enlightening example and details).

Furthermore, limit profiles such as \eqref{Reces} also appears in higher order convergence rates in periodic homogenization of fully nonlinear uniformly parabolic Cauchy problems with rapidly oscillating initial data, as shown below:
$$
\left\{
\begin{array}{rclcl}
  \frac{d}{dt}u^{\varepsilon}(x, t) & = & \frac{1}{\varepsilon^2}F(\varepsilon^2D^2 u^{\varepsilon}, x, t, \frac{x}{\varepsilon}, \frac{t}{\varepsilon}) & \text{in} & \R^n \times (0, T) \\
  u^{\varepsilon}(x, 0) & = & g\left(x, \frac{x}{\varepsilon}\right)& \text{on} & \R^n.
\end{array}
\right.
$$
In such a context
$$
\frac{1}{\varepsilon^2}F(\varepsilon^2 \mathrm{P}, x, t, y, s) \to F^{\sharp}( \mathrm{P}, x, t, y, s) \quad \text{as} \quad \varepsilon \to 0^+,
$$
uniformly for all $( \mathrm{P}, x, t, y, s) \in (\text{Sym}(n)\setminus \{\mathcal{O}_{n\times n}\})\times \R^n \times [0, T] \times \mathbb{T}^n \times \mathbb{T}$ (see \cite{KL20}). As a result, there exists a unique function $v: \R^n \times [0, T] \times \mathbb{T}^n \times [0, \infty) \to \R$ such that $v(x, t, \cdot, \cdot)$ is a viscosity solution to
$$
\left\{
\begin{array}{rclcl}
  \frac{d}{ds}v(y, s) & = & F^{\sharp}(D_y^2 v, x, t, y, s) & \text{in} & \mathbb{T}^n \times (0, \infty) \\
  v(x, t, y,  0) & = & g(y, x)& \text{on} & \mathbb{T}^n.
\end{array}
\right.
$$

\bigskip

To obtain our sharp Hessian estimates, we will assume that $F^{\sharp}$ enjoys good structural properties (e.g., convexity/concavity, or suitable \textit{a priori} estimates). Thus, using geometric tangential mechanisms, we may access good regularity estimates and, in a suitable manner, transfer such estimates to solutions of the original problem via compactness and stability processes.

It is important to emphasize that convexity assumptions with respect to second-order derivatives play an essential role in obtaining several regularity issues in fully nonlinear elliptic models, e.g., Calder\'{o}n-Zygmund estimates \cite{Caff1}, Schauder type estimates \cite{CC} and higher estimates in some nonlinear models \cite{Evans}, \cite{Kry82}, \cite{Tru83} and \cite{Tru84}, and free boundary problems \cite{daSV21-1}, \cite{daSV21}, \cite{FiShah14}, \cite{Indr19}, \cite{Indrei19}, \cite{IM16} and \cite{R-Oton-Serr17} just to mention a few.  Nevertheless, ensuring that such regularity issues hold true for problems like \eqref{E1} driven by operators satisfying the structure \eqref{Reces} is not an easy task. For instance, consider perturbations of Bellman operators of the form
{\scriptsize{
$$
\left\{
\begin{array}{rclcl}
\displaystyle F(\mathrm{X}, \varsigma, s, x) \defeq  \inf_{\iota \in \hat{\mathcal{A}}} \left(  \sum_{i,j=1}^{n} a^{\iota}_{ij}(x) \mathrm{X}_{ij} +  \sum_{i=1}^{n} b^{\iota}_i(x).\varsigma_i  + c^{\iota}(x)s\right) + \sum_{i=1}^{n} \textrm{arctg}(1+\lambda_i(\mathrm{X})) &=& f(x)& \mbox{in} &   \Omega \\
 \mathcal{B}(x,s,\varsigma)&=& g(x) &\mbox{on}& \partial \Omega,
\end{array}
\right.
$$}}
where $(\lambda_i(\mathrm{X}))^{n}_{i=1}$ are eigenvalues of the matrix $\mathrm{X} \in \text{Sym}(n)$, $b_i^{\iota}, c^{\iota}: \Omega \to \R$ are real functions for each $\iota \in \hat{\mathcal{A}}$ (for $\hat{\mathcal{A}}$ any set of index), and $\left( a^{\iota}_{ij}(x)\right)^{n}_{i,j=1}  \subset C^{0, \alpha}(\Omega)$ have eigenvalues in $[\lambda, \Lambda]$ for each $x \in \Omega$ and $\iota \in \hat{\mathcal{A}}$. Thus, it is easy to check that such operators are neither concave nor convex. However, its Recession profile \eqref{Reces} inherits good structure properties, i.e.,
  $$
  \left\{
  \begin{array}{rclcl}
    \displaystyle F^{\sharp}(D^2 \mathfrak{h}, D \mathfrak{h}, \mathfrak{h}, x)  & = &  \displaystyle \inf_{\iota \in \mathcal{A}} \left(  \sum_{i,j=1}^{n} a^{\iota}_{ij}(x) \mathfrak{h}_{ij}\right) = 0 & \text{in} & \Omega\\
      \mathcal{B}(x,\mathfrak{h}, D \mathfrak{h})&=& g(x) & \mbox{on} &  \partial \Omega,
  \end{array}
  \right.
  $$
which is a convex operator, thereby enjoying higher regularity estimates in the face of Evans-Krylov-Trudinger's type theory (with oblique boundary conditions) addressed in \cite[Theorem 2.4.]{DL20}, \cite[Theorem 1.3]{LiZhang}, \cite[Theorem 5.4]{Lieb02}, \cite[\S 8]{Saf91}, \cite[Theorem 3.3]{Saf94} and \cite[Theorem 8.3]{Sil1} (see also \cite[Chapter 6]{CC}, \cite{Evans}, \cite{Kry82}, \cite{Tru83} and \cite{Tru84} for local estimates), i.e., for some $\alpha \in (0, 1)$ and $\Omega^{\prime} \subset\subset \Omega \cup \Gamma$ - with $C^{1, \alpha}\ni \Gamma \subset \partial \Omega$, $\beta, \gamma, g \in C^{1, \alpha}(\overline{\Gamma})$, there holds
{\scriptsize{
\begin{equation}\label{AprioriEst}
  \|\mathfrak{h}\|_{C^{2, \alpha}(\overline{\Omega^{\prime}})} \le \mathrm{C}\left(n, \lambda, \Lambda, \alpha, \|a^{\iota}_{ij}\|_{C^{0, \alpha}(\Omega)}, \|\beta\|_{C^{1, \alpha}(\overline{\Gamma})}, \|\gamma\|_{C^{1, \alpha}(\overline{\Gamma})}, \Omega^{\prime}, \Omega\right)\left(\|\mathfrak{h}\|_{L^{\infty}(\Omega)}+\|g\|_{C^{1, \alpha}(\overline{\Gamma})}\right).
\end{equation}}}

Heuristically, if the ``limiting profile'' in \eqref{Reces} there exists point-wisely and fulfills a suitable \textit{a priori} estimate like \eqref{AprioriEst}, then a $W^{2, p}-$regularity theory to \eqref{E1} holds true up to the boundary via geometric tangential methods (cf. \cite{BOW16}, \cite{daSR19} and \cite{PT}).

Therefore, one of the main purposes of this work is to establish sharp Hessian estimates for problems enjoying a regular asymptotically elliptic property as \eqref{Reces} on the governing operator in \eqref{E1}, as well as present a number of applications in related analysis themes and free boundary problems of the obstacle-type.

Historically, the concept of an asymptotically regular problem can be traced back to Chipot-Evans' work \cite{Ch} in the context of certain problems in calculus of variations. Since then, there has been an increasing research interest in this class of problems and their related topics, see \cite{BJ0}, \cite{MF}, \cite{JP}, \cite{CS}, \cite{CS1} for an incomplete list of contributions. To the best of our knowledge, no research has been conducted on such boundary asymptotic qualitative estimates of general problems as \eqref{E1}, as well as their intrinsic connections and applications in the analysis of PDEs and free boundary problems.

\subsection{Assumptions and main results}

\hspace{0.4cm} For a given $r >0$, we denote by $\mathrm{B}_r(x)$ the ball of radius $r$ centered at $x=(x^{\prime},x_n) \in \R^n$, where $x^{\prime}=(x_1,x_2,\ldots, x_{n-1}) \in \R^{n-1}$. We write $\mathrm{B}_r = \mathrm{B}_r(0)$ and $\mathrm{B}^+_r = \mathrm{B}_r \cap \mathbb{R}^n_+$. Also, we write $\mathrm{T}_r \defeq \{(x^{\prime},0) \in \mathbb{R}^{n-1} : |x^{\prime}| < r\}$ and $\mathrm{T}_r(x_0) \defeq \mathrm{T}_r + x^{\prime}_0$ where $x^{\prime}_0 \in \mathbb{R}^{n-1}$. Finally, $\textrm{Sym}(n)$ will denote the set of all $n \times n$ real symmetric matrices.

Throughout this manuscript, we will make the following assumptions:
\begin{enumerate}
\item[(A1)] ({\bf Structural conditions}) We assume that $F \in C^0(\text{Sym}(n), \R^n, \R, \Omega)$. Moreover, there are constants $0 < \lambda \le \Lambda$, $\sigma\geq 0$ and $\xi\geq 0$ such that
\begin{eqnarray}\label{EqUnEll}
	\mathscr{P}^{-}_{\lambda,\Lambda}(\mathrm{X}-\mathrm{Y}) - \sigma |q_1-q_2| -\xi|r_1-r_2| &\le& F(\mathrm{X}, q_1,r_1,x)-F(\mathrm{Y},q_2,r_2,x) \nonumber \\
	&\le& \mathscr{P}^{+}_{\lambda, \Lambda}(\mathrm{X}-\mathrm{Y})+ \sigma |q_1-q_2| + \xi|r_1-r_2| \label{5}
\end{eqnarray}
for all $\mathrm{X},\mathrm{Y} \in \textit{Sym}(n)$, $q_1,q_2 \in \mathbb{R}^n$, $r_1,r_2 \in \mathbb{R}$, $x \in \Omega$, where
\begin{equation}
 	\mathscr{P}^{+}_{\lambda,\Lambda}(\mathrm{X}) \defeq  \Lambda \sum_{e_i >0} e_i +\lambda \sum_{e_i <0} e_i \quad \text{and} \quad \mathscr{P}^{-}_{\lambda,\Lambda}(\mathrm{X}) \defeq \Lambda \sum_{e_i <0} e_i + \lambda \sum_{e_i >  0} e_i,
 \end{equation}
 are the \textit{Pucci's extremal operators} and $e_i = e_i(\mathrm{X})$ ($1\leq i\leq n$) denote the eigenvalues of $\mathrm{X}$.

\item[(A2)] ({\bf Regularity of the data}) \, The data satisfy $f \in C^0(\Omega) \cap L^p(\Omega)$ for $n \le p<\infty$, $g, \gamma \in C^0(\partial \Omega)$ with $\gamma \le 0$ and $\beta \in C^0( \partial \Omega; \mathbb{R}^n)$ with $\|\beta\|_{L^{\infty}(\partial \Omega )} \le 1$ and there exists a positive constant $\mu_0$ such that $\beta\cdot \overrightarrow{\textbf{n}}\ge \mu_0$, where $\overrightarrow{\textbf{n}}$ is the outward normal vector of $\Omega$.

\item[(A3)] ({\bf Continuous coefficients in the $L^p$-average sense}) \,According to \cite{CCKS} for a fixed $x_0\in \Omega$, we define the quantity:
$$
	\psi_{F^{\sharp}}(x; x_0) \defeq \sup_{\mathrm{X} \in \textrm{Sym}(n)} \frac{|F^{\sharp}(\mathrm{X},0,0,x) - F^{\sharp}(\mathrm{X},0,0,x_0)|}{\|\mathrm{X}\|+1},
$$
which measures the oscillation of the coefficients of $F$ around $x_0$. Furthermore, for notation purposes, we shall often write $\psi_{F}(x, 0) = \psi_F(x)$.  Finally, the mapping $\Omega \ni x \mapsto F^{\sharp}(\mathrm{X},0,0,x)$ is assumed to be H\"{o}lder continuous function (in the $L^{p}$-average sense) for every $\mathrm{X} \in \textrm{Sym}(n)$. This means that, there exist universal constants\footnote{Throughout this work, a constant is said to be \textit{universal} if it depends only on $n, \lambda, \Lambda, p, \mu_0, \|\gamma\|_{C^{1,\alpha}(\partial \Omega)}$ and $\|\beta\|_{C^{1,\alpha}(\partial \Omega)}$} $\hat{\alpha} \in (0,1)$, $\theta_0 >0$ and $0 < r_0 \le 1$ such that
$$
\left( \intav{\mathrm{B}_r(x_0) \cap \Omega} \psi_{F^{\sharp}}(x,x_0)^{p} dx \right)^{1/p} \le \theta_0 r^{\hat{\alpha}}
$$
for $x_0 \in \overline{\Omega}$ and $0 < r \le r_0$.

\item[(A4)] ({\bf $C^{1,1}$ \textit{a priori} estimates}) We will assume that the recession operator $F^{\sharp}$ there exists and fulfills a up-to-the boundary $C^{1,1}$ \textit{a priori} estimates, i.e., for $x_0 \in \mathrm{B}^{+}_{1}$ and  $g_0 \in C^{1,\alpha}(\mathrm{T}_1)$ (for some $\alpha \in (0, 1)$), there exists a solution $\mathfrak{h} \in C^{1,1}(\mathrm{B}^+_1) \cap C^0(\overline{\mathrm{B}^+_1})$ of
$$
\left\{
\begin{array}{rclcl}
 F^{\sharp}(D^2 \mathfrak{h},x_0) &=& 0& \mbox{in} &   \mathrm{B}^+_1 \\
 \mathcal{B}(x,w,D\mathfrak{h})&=& g_0(x) &\mbox{on}& \mathrm{T}_1,
\end{array}
\right.
$$
such that
$$
	\|\mathfrak{h}\|_{C^{1,1}\left(\overline{\mathrm{B}^{+}_{\frac{1}{2}}}\right)} \le \mathrm{C}_1(\verb"universal") \left(\|\mathfrak{h}\|_{L^{\infty}(\mathrm{B}^{+}_1)}+\|g_0\|_{C^{1,\alpha}(\overline{\mathrm{T}_1})}\right)
$$
for some constant $\mathrm{C}_{1}>0$.
\end{enumerate}

 From now on, an operator fulfilling $\text{(A1)}$ will be referred to as a $(\lambda, \Lambda, \sigma, \xi)-$elliptic operator. Furthermore, we shall always assume the normalization condition: $ F(\mathcal{O}_{n\times n}, \overrightarrow{0}, 0, x) = 0 \quad \text{for all} \,\,\, x \in \Omega,$ which is not restrictive, because one can reduce the problem in order to check it.

We now present our first main result, which endures global $W^{2,p}$ estimates for viscosity solutions to asymptotically fully nonlinear equations.
\vspace{0.4cm}

\begin{theorem}[{\bf $W^{2, p}$ estimates under oblique boundary conditions}]\label{T1}
Let  $n \le p< \infty$, $\Omega \subset \mathbb{R}^n$, $\partial \Omega \in C^{2,\alpha}$, $\beta,\gamma,g \in C^{1,\alpha}(\partial \Omega)$ (for some $\alpha \in (0, 1)$) and $u$ be an $L^p-$viscosity solution of \eqref{E1}. Further, assume that structural assumptions (A1)-(A4) are in force. Then, there exists a constant $\mathrm{C}= \mathrm{C}(n, \lambda, \Lambda, p, \mu_0, \|\gamma\|_{C^{1,\alpha}(\partial \Omega)}, \|\beta\|_{C^{1,\alpha}(\partial \Omega)}, \|\partial \Omega\|_{C^{1,1}})>0$ such that   $u \in W^{2,p}(\Omega)$. Furthermore, the following estimate holds
\begin{equation} \label{2}
	\|u\|_{W^{2, p}(\Omega)} \le \mathrm{C}\cdot\left(\|u\|_{L^{\infty}(\Omega)}+ \|f\|_{L^p(\Omega)}+\| g\|_{C^{1,\alpha}(\partial \Omega)}   \right).
\end{equation}
\end{theorem}

\vspace{0.4cm}

 We also supply a sharp estimate in the scenario that $f$ exceeds the scope of $L^{\infty}$ functions. Precisely, our second result concerns $\text{BMO}$ type \textit{a priori} estimates for second derivatives provided the RHS of \eqref{E1} is a BMO function and $F^{\sharp}$ enjoys classical \textit{a priori} estimates.

In contrast to the assumptions on Section \ref{section2}, we will assume, for the next Theorem, the following hypothesis on $F^{\sharp}$:
\begin{enumerate}
		
	\item[(A4)$^{\star}$] ({\bf Higher \textit{a priori} estimates})\,\,$F^{\sharp}$ fulfills a $C^{2, \psi}$ \textit{a priori} estimate up to the boundary, for some $\psi \in (0,1)$, i.e., for any $g_0 \in C^0(\mathrm{B}^+_1) \cap C^{1, \psi}(\mathrm{T}_1)$, $\beta \in C^{1, \psi}(\mathrm{T}_1)$, there exists $\mathrm{C}_{\ast}>0$, depending only on universal parameters, such that any viscosity solution of
	$$
	\left\{
	\begin{array}{rclcl}
		F^{\sharp}(D^2 \mathfrak{h}, 0,0, x) &=& 0& \mbox{in} & \mathrm{B}^+_1,\\
		\beta \cdot D\mathfrak{h}  &=& g_0(x) & \mbox{on} &\mathrm{T}_1
	\end{array}
	\right.
	$$
	belongs to $C^{2, \psi}(\mathrm{B}^{+}_1) \cap C^0(\overline{\mathrm{B}^+_1})$ and fulfills
	$$
	\|\mathfrak{h}\|_{C^{2, \psi}\left(\overline{\mathrm{B}^+_{\rho}}\right)} \le\frac{\mathrm{C}_{\ast}}{\rho^{2+\psi}} \left(\|\mathfrak{h}\|_{L^{\infty}(\mathrm{B}^+_1)} + \|g_0\|_{C^{1,\psi}(\overline{\mathrm{T}_1})}\right)\,\,\, \forall \,\,0<\rho \ll 1.
	$$
\end{enumerate}

\begin{definition}
	For a function $f \in L^1_{\text{loc}}(\R^n)$ $x_0 \in \Omega$ and $\rho>0$, we define $(f)_{x_0, \rho}$ by
	$$
	(f)_{x_0, \rho} \defeq \intav{\mathrm{B}_{\rho}(x_0) \cap \Omega} f(x) dx.
	$$
	Moreover, when $x_0 = 0$ we will denote $(f)_{\rho}$ by simplicity.
	
	An $f \in L^1_{\text{loc}(\Omega)}$ is a function of $p-$bounded mean oscillation, in short $f \in \textrm{p-BMO}(\Omega)$, if
	\begin{equation}\label{p-BMOnorm}
		\|f\|_{p-BMO(\Omega)} \defeq \sup_{\mathrm{B}_{\rho} \subseteq \Omega \atop{x_0 \in \Omega}} \left(\intav{\mathrm{B}_{\rho}(x_0) \cap \Omega} |f(x) - (f)_{x_0, \rho}|^p dx\right)^{\frac{1}{p}} < \infty
	\end{equation}
	for a constant $\mathrm{C}>0$ that is independent on $\rho$. As a consequence of the John–Nirenberg inequality, such a semi-norm is equivalent to the one in the classical BMO spaces (see \cite[pp. 763-764]{PT}).
\end{definition}

\vspace{0.4cm}

	\begin{theorem}[{\bf Boundary $p-$BMO type estimates}]\label{BMO}
	Let $u$ be an $L^{p}$-viscosity solution to
$$
\left\{
\begin{array}{rclcl}
 F(D^2u,Du,u,x) &=& f(x)& \mbox{in} &   \mathrm{B}^+_1 \\
 \beta \cdot Du&=& g(x) &\mbox{on}& \mathrm{T}_1,
\end{array}
\right.
$$
where $f \in \textrm{p-BMO}(\mathrm{B}^+_1) \cap L^{p}(\mathrm{B}^+_1)$, for some $n \le p\leq \infty$. Assume further that assumptions (A1)-(A3) and $(A4)^{\star}$ are in force.
	Then, $D^2 u \in \text{p-BMO}(\mathrm{B}^+_1)$, Moreover, the following estimate holds
{\scriptsize{
\begin{equation}\label{BMO2}
		\|D^2 u\|_{p-\textrm{BMO}\left(\overline{\mathrm{B}^{+}_{\frac{1}{2}}}\right)} \le \mathrm{C}\left(n, \lambda, \Lambda, p, \mathrm{C}_{\ast}, \mu_0, \|\beta\|_{C^{1, \alpha}(\overline{\mathrm{T}_1})}\right)\left(\|u\|_{L^{\infty}(\mathrm{B}^+_1)} + \|g\|_{C^{1, \alpha}(\overline{\mathrm{T}_1})} + \|f\|_{\textrm{p-BMO}(\mathrm{B}^+_1)}\right).
\end{equation}}}
\end{theorem}

\begin{remark}[{\bf $C^{1, \text{Log-Lip}}$ estimates}] Since viscosity solutions to \eqref{E1} have a Hessian control in the $L^p-$ average sense, then by invoking the embedding result from \cite[Lemma 1]{AZBED}, we may establish $C^{1, \text{Log-Lip}}$ type estimates: Let $h: \mathrm{B}^{+}_{1} \rightarrow \mathbb{R}$ be a function such that for all $1 \le i, j\le n$, locally $D_{ij}h \in BMO_{p}(\mathrm{B}_{1}^{+})$. Then, $h \in C_{loc}^{1, \text{Log-Lip}}(\mathrm{B}_{1}^{+})$. In other words,
{\scriptsize{
$$
\displaystyle \sup_{\rho \in (0, 1)} \sup_{z \in \partial \Omega}\sup_{\mathrm{B}_{\rho}(z)\cap\Omega} \frac{|u(x)-[u(z)+D u(z)\cdot (x-z)]|}{\rho^2\ln \rho^{-1}} \leq \hat{\mathrm{C}}\cdot \left(\|u\|_{L^{\infty}(\Omega)} + \|g\|_{C^{1, \alpha}(\partial \Omega)} + \|f\|_{\textrm{p-BMO}(\Omega)}\right),
$$}}
for some constant $\hat{\mathrm{C}} = \hat{\mathrm{C}}(n, \lambda, \Lambda, p, \mathrm{C}_{\ast}, \mu_0, \|\beta\|_{C^{1, \alpha}(\overline{\mathrm{T}_1})},\|\partial \Omega\|_{C^{1,1}})$ (cf. \cite[Theorem 5.2]{daSR19} for related boundary estimates). Finally, we must quote \cite{IMN17}, where the authors prove essentially close-to-optimal assumptions on the $f(x,u)$, for which $C_{\text{loc}}^{1,1}$ regularity of $W^{2,p}$ solutions in the classical semi-linear equation $\Delta u = f(x,u)$ holds true.
\end{remark}

We would like to mention that although our manuscript has been strongly motivated by recent works \cite{BJ}, \cite{daSR19} and \cite{PT}, our approach has required a sort of non-trivial adaptation due to the presence of a non-homogeneous oblique boundary condition and the asymptotic nature of the limiting operator. In addition, unlike \cite{ZZZ21}, our findings provide additional quantitative applications, such as BMO type estimates (see Section \ref{Section5}), $W^{2,p}$--regularity to obstacle-type problems with oblique boundary conditions (see Section \ref{obst}), and density results for oblique type problems (see Section \ref{Sec_Density}) via tangential methods. 

Additionally, our recession profiles \ref{Reces} encode more general operators than linear ones, as addressed in \cite{ZZZ21} (see also \cite{BOW16}). Our results, in particular, extend, in the oblique boundary scenario, the former results from  \cite{BOW16}, \cite{PT} and \cite{ST} (see also \cite{daSR19}), and to some extent, those from \cite{BOW16}, \cite{Kry13}, \cite{Kry17} and \cite{ZZZ21} by employing techniques tailored to the general framework of the fully nonlinear models under relaxed convexity assumptions and oblique boundary datum.

Finally, we also highlight that Theorems \ref{T1} and \ref{BMO} can be extended to a more general class of models, including fully nonlinear parabolic PDEs. For sake of simplicity and clarity, we have decided to consider just the elliptic scenario in this paper. We pretend to study such issues in a forthcoming project.

\subsection{Applications to obstacle type problems}

\hspace{0.4cm} In closing, we would like to highlight that boundary Hessian estimates are also useful in the context of obstacle type problems. Physically, the classical obstacle problem refers to the equilibrium position of an elastic membrane (i.e., $u: \Omega \to \R$), whose boundary is held fixed (i.e., $u_{|\partial \Omega} = g$), lying above a given obstacle and subject to the action of external forces, \textit{e.g.}, friction, tension, and/or gravity ($f:\Omega \to \R$). More specifically, given an elastic membrane $u$ attached to a fixed boundary $\partial \Omega$ and an obstacle $\phi$, we seek the equilibrium position of the membrane when we move it downwards toward the obstacle, satisfying: 
$$
\left\{
\begin{array}{rcll}
  \Delta u & \le & f & \text{in the weak sense in} \quad \Omega \\
  \Delta u & = & f & \text{in the weak sense in} \quad \{u> \varphi\} \\
  u & \ge & \varphi & \text{in} \quad \Omega
\end{array}
\right.
$$
with $u-g \in H^1_0(\Omega)$.

We refer \cite{BLOP18} for Calder\'{o}n-Zygmund estimates for elliptic and parabolic obstacle problems in non-divergence form (under convex structure) with discontinuous coefficients and irregular obstacles. We must quote \cite{Lieb01} concerning gradient estimates for obstacle problems of elliptic models (under convex/convex structure) with oblique boundary conditions.

Now, we would like to focus our attention to the recent work \cite[Theorem 1.1.]{BJ1}, where $W^{2,p}$ estimates for the obstacle problem with oblique boundary conditions
\begin{equation} \label{obss1}
	\left\{
	\begin{array}{rclcl}
		F(D^2 u,Du,u,x) &\le& f(x)& \mbox{in} &   \Omega \\
		(F(D^2 u, Du, u,x) - f(x))(u(x)-\phi(x)) &=& 0 &\mbox{in}& \Omega\\
		u(x) &\ge& \phi(x) &\mbox{in}& \Omega\\
		\beta(x) \cdot Du(x)  + \gamma(x)u(x) &=& g(x) &\mbox{on}& \partial \Omega\\
	\end{array}
	\right.
\end{equation}
for a given obstacle $\phi \in W^{2,p}(\Omega)$ satisfying $\beta(x) \cdot D \phi(x) \ge g(x)$ a.e., on $\partial \Omega$, have been obtained in the case when $F$ is a convex operator and $\gamma  = g \equiv 0$.

Therefore, motivated by above problem, we will investigate $W^{2,p}$ regularity estimates for obstacle-type problems \eqref{obss1} under an asymptotic convexity assumption (weaker than convexity) and non-homogeneous oblique conditions.

For our purpose, it is necessary to make the following assumptions:
\begin{enumerate}
\item[(A5)] There exists a modulus of continuity $\omega: [0,+\infty) \to [0,+\infty)$ with $\omega(0)=0$, such that
	$$
	F(\mathrm{X}_1, q, r,x_1) - F(\mathrm{X}_2,q,r,x_2) \le \omega\left(|x_1-x_2|\right)\left[(|q| +1) + \alpha_0 |x_1-x_2|^2\right]
	$$
	holds for any $x_1,x_2 \in \Omega$, $q \in \mathbb{R}^n$, $r \in \mathbb{R}$, $\alpha_0 >0$ and $\mathrm{X}_1,\mathrm{X}_2 \in \textrm{Sym}(n)$ satisfying
$$
		- 3 \alpha_0
		\begin{pmatrix}
			\mathrm{Id}_n& 0 \\
			0& \mathrm{Id}_n
		\end{pmatrix}
		\leq
		\begin{pmatrix}
			\mathrm{X}_2&0\\
			0&-\mathrm{X}_1
		\end{pmatrix}
		\leq
		3 \alpha_0
		\begin{pmatrix}
			\mathrm{Id}_n & -\mathrm{Id}_n \\
			-\mathrm{Id}_n& \mathrm{Id}_n
		\end{pmatrix},	
$$
	where $\mathrm{Id}_n$ is the identity matrix.
	\item[(A6)] $F$ is a proper operator in the following sense:
	$$
d\cdot (r_{2}-r_{1}) \leq F(\mathrm{X},q,r_{1},x)-F(\mathrm{X},q,r_{2},x),
	$$
	for any $\mathrm{X} \in \text{Sym}(n)$, $r_1,r_2 \in \mathbb{R}$, with $r_{1}\leq r_{2}$, $x \in \Omega$, $q \in \mathbb{R}^n$, and some $d >0$.
\end{enumerate}

The previous assumptions are invoked to assure the validity of the Comparison Principle for oblique derivatives
problems like \eqref{E1} (see \cite[Theorem 2.10]{CCKS} and \cite[Theorem 7.17]{Leiberman}), ensuring the use of Perron's Method for viscosity solutions (see Lieberman's Book \cite[Chapter 7.4 and 7.6]{Leiberman}).

Finally, our third result addresses a $W^{2,p}$-regularity theory for \eqref{obss1}, that does not rely on any additional regularity assumptions on nonlinearity $F$.

\begin{theorem}[{\bf Obstacle problems with \textit{oblique} boundary conditions}]\label{T3}
	Let $u$ be an $L^p$-viscosity solution of \eqref{obss1} with $n<p< \infty$, where $F$ satisfies the structural assumption (A1)-(A4) and (A5)-(A6), $\partial \Omega \in C^{3}$, $f \in L^p(\Omega)$, $\beta \in C^{2}(\partial \Omega)$ with $\beta \cdot {\bf \overrightarrow{\bf{n}}} \ge \mu_0$ for some $\mu_0 >0$, $\phi \in W^{2,p}(\Omega)$ and $g \in C^{1,\alpha}(\partial \Omega)$ (for some $\alpha \in (0, 1)$). Then, $u \in W^{2,p}(\Omega)$ with the following estimate
	\begin{equation} \label{obs2}
		\|u\|_{W^{2, p}(\Omega)} \le \mathrm{C}_0\cdot \left( \|f\|_{L^p(\Omega)}+ \|\phi\|_{W^{2,p}(\Omega)}  +\|g\|_{C^{1,\alpha}(\partial \Omega)} \right).
	\end{equation}
where $\mathrm{C}_0 = \mathrm{C}_0(n,\lambda,\Lambda, p,\mu_0, \sigma, \omega, \|\beta\|_{C^2(\partial \Omega)}, \|\gamma\|_{C^2(\partial \Omega)}, \partial \Omega, \textrm{diam}(\Omega), \theta_0)$.
\end{theorem}

In conclusion, we must stress that similar Hessian estimates hold true for obstacle problems with Dirichlet boundary conditions:
$$
	\left\{
	\begin{array}{rclcl}
		F(D^2 u,Du,u,x) &\le& f(x)& \mbox{in} &   \Omega \\
		(F(D^2 u, Du, u,x) - f(x))(u(x)-\phi(x)) &=& 0 &\mbox{in}& \Omega\\
		u(x) &\ge& \phi(x) &\mbox{in}& \Omega\\
		u(x)   &=& g(x) &\mbox{on}& \partial \Omega
	\end{array}
	\right.
$$
by invoking the related global $W^{2, p}$ regularity estimates addressed in \cite[Theorem 2.7]{BOW16} and \cite[Theorems 1.1]{daSR19} with the following estimate
$$
 \|u\|_{W^{2, p}(\Omega)} \le \mathrm{C}(\verb"universal")\cdot \left( \|f\|_{L^p(\Omega)}+ \|\phi\|_{W^{2,p}(\Omega)}  +\|g\|_{W^{2,p}(\partial \Omega)} \right) \quad \text{for} \quad n<p< \infty.
$$

Finally, we would like to mention the interesting manuscript \cite{Indr19}, in which the author utilized the regularity (in the case $p=\infty$) to prove a conjecture on the geometry of the free and fixed boundaries in the setting of fully nonlinear obstacle problems (cf. also \cite{Indrei19} for related regularity results).

\subsection{Hessian estimates to models in non-divergence form: State-of-Art }

\hspace{0.4cm} Regularity estimates in Sobolev spaces for nonlinear models have instituted an important chapter in the modern history of the elliptic theory of viscosity solutions. Moreover, because of its non-variational nature, obtaining sharp integrability properties of solutions of such elliptic operators have been a demanding task throughout the last years. One of the most pressing issues was whether $W^{2,p}$ a priori estimates could be addressed for any fully nonlinear operator, thereby attracting notable efforts from the academic community over the last three decades or so, while remaining open to the advancement of modern techniques.

In this direction, Caffarelli proved in his seminal work \cite[Theorem 7.1]{Caff1} that $C^0-$viscosity solutions to
  \begin{equation}\label{EqCaf}
    F(D^2 u, x) = f(x) \quad \text{in} \quad \mathrm{B}_1
  \end{equation}
satisfy an interior $W^{2,p}-$regularity for $n < p < \infty$,  with the following estimate
  $$
  	\|u\|_{W^{2, p}\left(\mathrm{B}_{\frac{1}{2}}\right)} \le \mathrm{C}(n, p, \lambda, \Lambda )\left(\|u\|_{L^{\infty}(\mathrm{B}_1)} + \|f\|_{L^p(\mathrm{B}_1)}\right),
  $$
  provided a small oscillation of $F(\mathrm{X},x)$ in the variable $x$ and $C^{1,1}$ \textit{a priori} estimates for homogeneous equation with ``frozen coefficients'' $F(D^2 w, x_0)=0$ are in force. Furthermore, Caffarelli also demonstrated that for any $p<n$, there exists a fully nonlinear operator fulfilling the previous hypothesis, for which $W^{2,p}-$estimates fail to hold. In this regard, Dong-Kim's work (cite{DK91}) is worth mentioning for notable examples of linear elliptic operators in non-divergence form with piecewise constant coefficients, whose solutions lack $W^{2, p}$ estimates.

  A few years later, Caffarelli's $W^{2, p}-$estimates were extended by Escauriaza in \cite[Theorem 1]{Es93}
  in the context of $L^n-$viscosity solutions to the range of exponents $n - \varepsilon_0 < p < \infty$ with $\varepsilon_0  = \varepsilon_0\left(\frac{\Lambda}{\lambda}, n\right) \in \left(\frac{n}{2}, n\right)$ (the Escauriaza's constant), which provides the minimal range of integrability for which the Harnack inequality (resp. H\"{o}lder regularity) holds for viscosity solutions to \eqref{EqCaf}. Furthermore, when the boundary of the domain is smooth enough, namely $C^{1,1}$, a global $W^{2,p}$ estimate is obtained for $n - \varepsilon_0< p < \infty$ by Winter in \cite[Theorem 4.5]{Winter}.

  In the aftermath, Byun-Lee-Palagachev in \cite{BLP} extend the Escauriaza-Winter's results by studying the global regularity of viscosity solutions to the following Dirichlet problem for a fully nonlinear uniformly elliptic equation:
\begin{equation}\label{Eq-Weight}
  \left\{
\begin{array}{rclcl}
 F(D^2u,Du,u,x) &=& f(x)& \mbox{in} &   \Omega \\
 u(x) &=& 0 &\mbox{on}& \partial \Omega
\end{array}
\right.
\end{equation}
in weighted Lebesgue spaces. Indeed, under appropriate structural conditions, if $f\in L^p_w(\Omega)$ for $p>n-\varepsilon_0$ and $w \in \mathcal{A}_{\frac{p}{n-\varepsilon_0}}$ (Muckenhoupt classes), then the problem \eqref{Eq-Weight} admits a unique solution $u \in W^{2,p}_w(\Omega)$ such that
$$
\|u\|_{W^{2, p}_w\left(\Omega\right)} \le \mathrm{C}(\verb"universal")\|f\|_{L^p_w(\Omega)}.
$$

  Recently, in a new landmark for this theory, Caffarelli's local $W^{2, p}$ regularity estimates were extended by Pimentel-Teixeira in \cite{PT}. Precisely, the novelty with respect to the Caffarelli's results is the concept of \textit{Recession} operator (a sort of tangent profile for $F$ at ``infinity'', see \eqref{Reces}). In this context, the authors relaxed the hypothesis of $C^{1,1}$ \textit{a priori} estimates for solutions of equations without dependence on $x$, by assuming that $F$ must be ``convex or concave'' only at the ends of $\textit{Sym}(n)$. In this setting, the authors proved that any viscosity solution to
 \begin{equation}\label{Eq-Asymp-Local}
    	F(D^2 u) =f(x) \quad \textrm{in} \quad \mathrm{B}_1,
 \end{equation}
 where $f \in L^{p}(\mathrm{B}_1) \cap C^0(\mathrm{B}_1)$, fulfills $u \in W^{2,p}_{loc}(\mathrm{B}_1)$ with an \textit{a priori} estimate
 $$
	\|u\|_{W^{2, p}\left(\mathrm{B}_{\frac{1}{2}}\right)} \le \mathrm{C}(n, p, \lambda, \Lambda)\left(\|u\|_{L^{\infty}(\mathrm{B}_1)} + \|f\|_{L^p(\mathrm{B}_1)}\right).
$$

We must quote similar global $W^{2, p}-$estimates independently proved by Byan \textit{et al}. in \cite[Theorems 2.5 and 2.7]{BOW16} for solutions to asymptotically linear operators with a zero Dirichlet boundary condition. Their approach is based on a proper transformation that converts a given asymptotically elliptic PDE to a suitable uniformly elliptic one. We also cite Silvestre-Teixeira's work \cite[Theorems 1 and 4]{ST} for an interesting survey on this theme in the setting of sharp gradient estimates.

In the sequel, global $W^{2,p}$, $\text{BMO}_p$ and $C^{1, \text{Log-Lip}}$ regularity estimates, in the context of asymptotic convex operators (see Definition \eqref{DefAC}), i.e.,
$$
    \left\{
\begin{array}{rclcl}
 F(D^2u,Du,u,x) &=& f(x)& \mbox{in} &   \Omega\\
 u(x) &=& g(x) &\mbox{on} & \partial \Omega
\end{array}
\right.
$$
were addressed by da Silva-Ricarte in \cite[Theorems 1.1, 1.2 and 5.2]{daSR19}. Precisely, they obtain
$$
\displaystyle   \|u\|_{W^{2, p}\left(\Omega\right)} \le \mathrm{C}(n, p, \lambda, \Lambda, \|\partial \Omega\|_{C^{1, 1}})\left(\|u\|_{L^{\infty}(\Omega)} + \|f\|_{L^p(\Omega)} + \|g\|_{W^{2,p}(\partial \Omega)}\right).
$$
The proof of such estimates in \cite{daSR19} is based on the successful adaptation of compactness arguments inspired by the ideas from Pimentel-Teixeira \cite{PT}, Silvestre-Teixeira \cite{ST}, and Winter \cite{Winter}. We also refer the reader to Lee's work \cite{Lee19} for Hessian estimates for solutions of \eqref{Eq-Asymp-Local} in the framework of weighted Orlicz spaces.

Among other works on this research topic, we must cite the sequence of Krylov's fundamental manuscripts \cite{Kry12}, \cite{Kry13} and \cite{Kry17}, which are associated with the existence of $L^p-$viscosity
solutions under either relaxed or no convexity assumptions on $F$.

Now, concerning problems with oblique boundary conditions, Hessian estimates for viscosity solutions of elliptic equations like \eqref{E1} has been regarded as an important issue in the study of PDEs for a long time. The theory of this subject has also been improved over the last several decades. For general fully nonlinear problems like \eqref{E1}, the existence and uniqueness for viscosity solutions of oblique boundary problems was proved in \cite{Hi} and \cite{Leiberman}.

Inspired by the ideas from Caffarelli's fundamental work \cite{Caff1}, Byan and Han in \cite{BJ} establish global $W^{2,p}-$estimates for viscosity solutions of convex fully nonlinear elliptic models of the form:
$$
\left\{
\begin{array}{rclcl}
 F(D^2u,Du,u,x) &=& f(x)& \mbox{in} &   \Omega \\
 \beta(x) \cdot Du(x) &=& 0 &\mbox{on}& \partial \Omega.
\end{array}
\right.
$$
Moreover, the following estimate holds
$$
	\|u\|_{W^{2,p}(\Omega)} \le C(\verb"universal")\left( \|u\|_{L^{\infty}(\Omega)} + \|f\|_{L^p(\Omega)} \right),
$$
It is also worth highlighting that $C^{0, \alpha}$, $C^{1,\alpha}$ and $C^{2, \alpha}$ boundary regularity were addressed for the Neumann problem on flat boundaries by Milakis-Silvestre in \cite{Sil1}. Moreover, such results were recently extended to the general oblique scenario by Li-Zhang in \cite{LiZhang}.

Finally, we must refer to the work of Zhang \textit{et al.} \cite{ZZZ21}, which establishes $W^{2,p}-$regularity for viscosity solutions of fully nonlinear elliptic equations with a homogeneous oblique derivative boundary condition. In this context, the nonlinearity is assumed to fulfill a certain asymptotic regularity condition.

In conclusion, taking such historical advances into account, in this work, we aim to obtain a global $W^{2,p}$ estimates for a large class of nonlinear elliptic PDEs. To that end, we only assume that $F(\mathrm{X}, \varsigma, s, x)$ is asymptotically elliptic, that is, we establish global Sobolev \textit{a priori} estimates for \eqref{E1} solutions under a relaxed convexity assumption (a suitable asymptotic structure of $F$ just at "infinity" of $\textit{Sym}(n)$ according to the Definition \ref{DefAC}).

This paper is organized as follows: Section \ref{section2} contains the main notations and preliminary results on which we will work throughout this manuscript. In Section \ref{section3}, we present the \textit{Modus Operandi}, i.e., the tangential approximation mechanism used in order to relate the estimates coming from \textit{Recession} operator $F^{\sharp}$ with its original counterpart $F$. Section \ref{Sec4} is devoted to proving Theorem \ref{T1}. In Section \ref{Section5} we establish the regularity estimates in the borderline setting, namely for $\textrm{p-BMO}$ spaces. Finally, in the final sections, we establish the global $W^{2,p}$ estimate for obstacle problems with oblique boundary conditions, and a density result in a suitable class of viscosity solutions.


\section{Preliminaries and auxiliary results} \label{section2}

\hspace{0.4cm}To \eqref{E1}, we introduce the appropriate notions of viscosity solutions.

\begin{definition}[{\bf $C^{0}-$viscosity solutions}]\label{VSC_0} Let $F$ be a $\left(\lambda, \Lambda, \sigma, \xi\right)-$elliptic operator and $\Gamma \subset \partial\Omega$ a relatively open set. A function $u \in C^0(\overline{\Omega})$  is said a $C^{0}-$viscosity solution  if the following condition hold:
\begin{enumerate}
\item[a)] for all $ \forall\,\, \varphi \in C^{2} (\Omega \cup \Gamma)$ touching $u$ by above  at  $x_0 \in \Omega \cup \Gamma$,
$$
\left\{
\begin{array}{rcl}
  F\left(D^2 \varphi(x_{0}), D \varphi(x_{0}), \varphi(x_{0}), x_{0}\right)  \geq f(x_0) & \text{when} & x_0 \in \Omega \\
  \text{and} \quad \mathcal{B}(x_0, u(x_0), D \varphi(x_0)) \ge g(x_0) & \text{at} & x_0 \in \Gamma.
\end{array}
\right.
$$
\item[b)] for all $ \forall\,\, \varphi \in C^{2 } (\Omega \cup \Gamma)$ touching $u$ by below  at  $x_0 \in \Omega \cup \Gamma$,
$$
\left\{
\begin{array}{rcl}
  F\left(D^2 \varphi(x_{0}), D \varphi(x_{0}), \varphi(x_{0}), x_{0}\right)  \leq f(x_0) & \text{when} & x_0 \in \Omega \\
  \text{and} \quad \mathcal{B}(x_0, u(x_0), D \varphi(x_0)) \le g(x_0) & \text{at} & x_0 \in \Gamma.
\end{array}
\right.
$$

\end{enumerate}
\end{definition}
\begin{definition}[{\bf$L^{p}$-viscosity solution}]\label{VSLp}
Let $F$ be a $(\lambda,\Lambda,\sigma, \xi)$-elliptic operator, $p\ge n$ and $f\in L^{p}(\Omega)$. Assume that $F$ is continuous in $\mathrm{X}$, $q$, $r$ and measurable in $x$. A function $u\in C^0(\overline{\Omega})$ is said an $L^{p}$-viscosity solution for $\ref{E1}$ if the following assertions hold:
\begin{enumerate}
\item [a)] For all $\varphi\in W^{2,p}(\overline{\Omega})$ touching $u$ by above  at  $x_0 \in \overline{\Omega}$
$$
\left\{
\begin{array}{rcl}
	F\left(D^2 \varphi(x_{0}), D \varphi(x_{0}), \varphi(x_{0}), x_{0}\right)  \geq f(x_0) & \text{when} & x_0 \in \Omega \\
	\text{and} \quad \mathcal{B}(x_0, u(x_0), D \varphi(x_0)) \ge g(x_0) & \text{at} & x_0 \in \partial \Omega.
\end{array}
\right.
$$
\item [b)] For all $\varphi\in W^{2,p}(\overline{\Omega})$ touching $u$ by below  at  $x_0 \in \overline{\Omega}$
$$
\left\{
\begin{array}{rcl}
	F\left(D^2 \varphi(x_{0}), D \varphi(x_{0}), \varphi(x_{0}), x_{0}\right)  \leq f(x_0) & \text{when} & x_0 \in \Omega \\
	\text{and} \quad \mathcal{B}(x_0, u(x_0), D \varphi(x_0)) \le g(x_0) & \text{at} & x_0 \in \partial \Omega.
\end{array}
\right.
$$
\end{enumerate}
\end{definition}

For convenience of notation, we define
$$
	\mathcal{L}^{\pm}(u) \defeq \mathscr{P}^{\pm}_{\lambda,\Lambda}(D^2 u) \pm \sigma |Du|.
$$
\begin{definition}
 We define the class $\overline{\mathcal{S}}\left(\lambda,\Lambda,\sigma, f\right)$ and $\underline{S}\left(\lambda,\Lambda,\sigma, f\right)$ to be the set of all continuous functions $u$ that satisfy $\mathcal{L}^{+}(u) \ge f$, respectively $\mathcal{L}^{-}(u) \le f$ in the viscosity sense (see Definition \ref{VSC_0}). We define,
 $$
 	\mathcal{S}\left(\lambda, \Lambda, \sigma,f\right) \defeq  \overline{\mathcal{S}}\left(\lambda, \Lambda, \sigma,f\right) \cap \underline{\mathcal{S}}\left(\lambda, \Lambda,\sigma, f\right)\,\,\text{and}\,\,
 \mathcal{S}^{\star}\left(\lambda, \Lambda,\sigma, f\right) \defeq  \overline{\mathcal{S}}\left(\lambda, \Lambda, \sigma,|f|\right) \cap \underline{\mathcal{S}}\left(\lambda, \Lambda,\sigma, -|f|\right).
  $$
  Moreover, when $\sigma=0$, we denote $\mathcal{S}^{\star}(\lambda,\Lambda,0,f)$ just by $S^{\star}(\lambda,\Lambda,f)$ (resp. by $\underline{S}, \overline{S}, \mathcal{S}$).
 \end{definition}

Now, we present a compactness result, whose proof can be found in \cite[Theorem 1.1]{LiZhang}.

\begin{theorem}[{\bf $C^{0, \alpha^{\prime}}$ regularity}]\label{Holder_Est} Let $u \in C^0(\Omega) \cap C^0(\Gamma)$ be satisfying
$$
\left\{
\begin{array}{rcl}
  u \in \mathcal{S}(\lambda, \Lambda, f) & \text{in} & \Omega \\
  \mathcal{B}(x, u, Du)= g(x)& \text{on} & \Gamma.
\end{array}
\right.
$$
Then, for any $\Omega^{\prime} \subset\subset \Omega \cup \Gamma$, $u \in C^{0, \alpha^{\prime}}(\overline{\Omega^{\prime}})$ and
$$
\|u\|_{C^{0, \alpha^{\prime}}(\overline{\Omega^{\prime}})} \le \mathrm{C}(n, \lambda,\Lambda, \mu_0, \|\gamma\|_{L^{\infty}(\Gamma)}, \Omega^{\prime}, \Omega)\left(\|u\|_{L^{\infty}(\Omega)}+ \|f\|_{L^n(\Omega)}+\|g\|_{L^{\infty}(\Gamma)}\right)
$$
where $\alpha^{\prime} \in (0, 1)$ depends on $n, \lambda,\Lambda$ and $\mu_0$.
\end{theorem}

Next, we present the following Stability result (see \cite[Theorem 3.8]{CCKS} for a proof).
\begin{lemma}[{\bf Stability Lemma}]\label{Est}
For $k \in \mathbb{N}$ let $\Omega_k \subset \Omega_{k+1}$ be an increasing sequence of domains and $\displaystyle \Omega \defeq \bigcup_{k=1}^{\infty} \Omega_k$. Let $p \ge n$ and $F, F_k$ be $(\lambda, \Lambda, \sigma, \xi)-$elliptic operators. Assume $f \in L^{p}(\Omega)$, $f_k \in L^p(\Omega_k)$ and that $u_k \in C^0(\Omega_k)$ are $L^{p}-$viscosity sub-solutions (resp. super-solutions) of
$$
	F_k(D^2 u_k,Du_k,u_k,x)=f_k(x) \quad \textrm{in} \quad \Omega_k.
$$
Suppose that $u_k \to u_{\infty}$ locally uniformly in $\Omega$ and that for $\mathrm{B}_r(x_0) \subset \Omega$ and $\varphi \in W^{2,p}(\mathrm{B}_r(x_0))$ we have
\begin{equation} \label{Est1}
	\|(\hat{g}-\hat{g}_k)^+\|_{L^p(\mathrm{B}_r(x_0))} \to 0 \quad \left(\textrm{resp.} \,\,\, \|(\hat{g}-\hat{g}_k)^-\|_{L^p(\mathrm{B}_r(x_0))} \to 0 \right),
\end{equation}
where $\hat{g}(x) \defeq F(D^2 \varphi, D \varphi, u,x)-f(x)$ and $\hat{g}_k(x) =  F_k(D^2 \varphi, D \varphi, u_{k},x)-f_k(x)$.  Then, $u$ is an $L^{p}-$viscosity sub-solution (resp. super-solution) of
$$
	F(D^2 u,Du,u,x)=f(x) \quad \textrm{in} \quad \Omega.
$$
Moreover, if $F, f$ are continuous, then $u$ is a $C^0-$viscosity sub-solution (resp. super-solution) provided that \eqref{Est1} holds for all $\varphi \in C^2(\mathrm{B}_r(x_0))$ test function.
\end{lemma}

The next result is an A.B.P. Maximum Principle (see \cite[Theorem 2.1]{LiZhang} for a proof).

\begin{lemma}[{\bf A.B.P. Maximum Principle}]\label{ABP-fullversion}
	Let $\Omega\subset \mathrm{B}_{1}$ and $u \in C^{0}(\overline{\Omega})$ satisfying
	\begin{equation*}
		\left\{
		\begin{array}{rclcl}
			u\in \mathcal{S}(\lambda,\Lambda,f) &\mbox{in}&   \Omega \\
			 \mathcal{B}(x,u,Du)=g(x)  &\mbox{on}&  \Gamma.
		\end{array}
		\right.
	\end{equation*}
	Suppose that exists $\varsigma\in \partial \mathrm{B}_{1}$ such that $\beta\cdot\varsigma\geq \mu_0$ and $\gamma\le 0$ in $\Gamma$. Then,
	\begin{eqnarray*}
		\|u\|_{L^{\infty}(\Omega)}\leq \|u\|_{L^{\infty}(\partial \Omega\setminus \Gamma)}+\mathrm{C}(n, \lambda, \Lambda, \mu_0)(\| g\|_{L^{\infty}(\Gamma)}+\| f\|_{L^{n}(\Omega)})
	\end{eqnarray*}
\end{lemma}

In the sequel, we comment on the existence and uniqueness of viscosity solutions with oblique boundary conditions. For that purpose, we will assume the following condition on $F$:

\begin{enumerate}
	\item [$(\bf SC)$]  There exists a modulus of continuity $\tilde{\omega}$, i.e., $\tilde{\omega}$ is nondecreasing with $ \displaystyle \lim_{t \to 0} \tilde{\omega}(t) =0$ and
	$$
	\psi_{F}(x,y) \le \tilde{\omega}(|x-y|).
	$$
	
\end{enumerate}

 Now, we will ensure the existence and uniqueness of viscosity solutions to the following problem:
$$
\left\{
\begin{array}{rclcl}
	F(D^2u, x) &=& f(x) & \mbox{in} & \Omega,\\
	 \mathcal{B}(x,u,Du)&=& g(x) & \mbox{on} & \Gamma\\
	u(x)&=&\varphi(x) & \mbox{on} & \partial \Omega\setminus \Gamma,
\end{array}
\right.
$$
where $\Gamma$ is relatively open of $\partial\Omega$. The next proof follows the ideas from \cite[Theorem 3.1]{LiZhang} with minor modifications. For this reason, we will omit it here.

\begin{theorem}\label{comparation}
	Suppose that $\Gamma\in C^{2}$ and $\beta\in C^{2}(\overline{\Gamma})$. Assume that $F$ satisfies (A1) and $(SC)$ . Let $u$ and $v$ be satisfying
	$$
	\left\{
	\begin{array}{rclcl}
		F(D^2u, x) &\geq& f_{1}(x) & \mbox{in} & \Omega,\\
		 \mathcal{B}(x,u,Du)&\geq& g_{1}(x) & \mbox{on} & \Gamma\\
	\end{array}
	\right.
	$$
	and
	$$
	\left\{
	\begin{array}{rclcl}
		F(D^2v, x) &\leq& f_{2}(x) & \mbox{in} & \Omega,\\
		 \mathcal{B}(x,v,Dv)&\leq& g_{2}(x) & \mbox{on} & \Gamma.\\
	\end{array}
	\right.
	$$
	Then,
	$$
	\left\{
	\begin{array}{rclcl}
		u-v\in \underline{\mathcal{S}}\left(\frac{\lambda}{n},\Lambda,f_{1}-f_{2}\right) & \mbox{in} & \Omega,\\
	 \mathcal{B}(x,u-v,D(u-v))\geq (g_{1}-g_{2})(x) & \mbox{on} & \Gamma.\\
	\end{array}
	\right.
	$$
\end{theorem}

\begin{theorem}[{\bf Uniqueness}]\label{Unicidade}
	Let $\Gamma\in C^{2}$, $\beta\in C^{2}(\overline{\Gamma})$, $\gamma\le 0$ and $\varphi\in C^{0}(\partial \Omega\setminus \Gamma)$. Suppose that there exists $\varsigma\in\partial \mathrm{B}_{1}$ such that $\beta\cdot \varsigma\ge \mu_0$ on $\Gamma$.  Assume that $F$ satisfies $(A1)$. Then, there exists at most one viscosity solution of
	$$
	\left\{
	\begin{array}{rclcl}
F(D^2u, x) &=& f(x) & \mbox{in} & \Omega,\\
\mathcal{B}(x,u,Du)&=& g(x) & \mbox{on} & \Gamma\\
u(x)&=&\varphi(x) & \mbox{on} & \partial \Omega\setminus \Gamma.
	\end{array}
	\right.
	$$
\end{theorem}
\begin{proof}
	The proof is direct consequence of Theorem \ref{comparation} with A.B.P. estimate (Lemma \ref{ABP-fullversion}).
\end{proof}

\begin{theorem}[{\bf Existence}]\label{Existencia}
	Let $\Gamma\in C^{2}$, $\beta\in C^{2}(\overline{\Gamma})$, $\gamma\le 0$ and $\varphi\in C^{0}(\partial \Omega\setminus \Gamma)$. Suppose that there exists $\varsigma\in\partial \mathrm{B}_{1}$ such that $\beta\cdot \varsigma\ge \mu_0$ on $\Gamma$ and assume $(SC)$. In addition, suppose that $\Omega$ satisfies an exterior cone condition at any $x\in \partial \Omega\setminus \overline{\Gamma}$ and satisfies an exterior sphere condition at any $x\in \overline{\Gamma}\cap (\partial \Omega\setminus \Gamma)$. Then, there exists a unique viscosity solution $u\in C^{0}(\overline{\Omega})$ of
	$$
	\left\{
	\begin{array}{rclcl}
		F(D^2u, x) &=& f(x) & \mbox{in} & \Omega,\\
		 \mathcal{B}(x,u,Du)&=& g(x) & \mbox{on} & \Gamma\\
		u(x)&=&\varphi(x) & \mbox{on} & \partial \Omega\setminus \Gamma.
	\end{array}
	\right.
	$$
\end{theorem}
\begin{proof}
The proof follows the same lines as the one in \cite[Theorem 3.3]{LiZhang} with minor modifications. For instance, we must invoke Theorem \ref{Unicidade} instead of \cite[Theorem 3.2]{LiZhang}. Other necessary modifications are the same made in the proof of Theorem \ref{comparation}.
\end{proof}

Next result is a key tool in our tangential approach. Precisely, it describes how the point-wise convergence of $F_{\tau}$  to $F^{\sharp}$ takes place.

\begin{lemma}
Let $F$ be a uniformly elliptic operator and assume $F^{\sharp}$ there exists. Then, given $\epsilon >0$ there exists $\tau_0(\lambda, \Lambda, \epsilon, \psi_{F^{\ast}}) >0$ such that, for every $\tau  \in (0, \tau_0)$ there holds
$$
	\frac{\left|\tau F \left(\tau^{-1} \mathrm{X}, 0, 0, x\right) - F^{\sharp}(\mathrm{X}, 0, 0, x) \right|}{1+\|\mathrm{X}\|} \le \epsilon,
$$
for every $\mathrm{X} \in \text{Sym}(n)$.
\end{lemma}

\begin{proof}
The proof follows the same lines as the one in \cite{ST} (see also \cite{PT}). For this reason, we omit it here.
\end{proof}

\begin{remark} It is worth highlighting that the $W^{2, p}$ estimates from Theorem \ref{T1} depends not only on universal constants, but in effect on the ``modulus of convergence'' $F_{\tau} \to F^{\sharp}$. Precisely, by defining $\varsigma:(0, \infty) \to (0, \infty)$ as follows
$$
\displaystyle  \varsigma(\varepsilon) \defeq \sup_{\mathrm{X} \in \text{Sym}(n) \atop{\tau \in (0, \tau_0(\varepsilon))}} \left\{\frac{\left|\tau F\left(\tau^{-1} \mathrm{X}, 0, 0, x\right) - F^{\sharp}(\mathrm{X}, 0, 0, x) \right|}{1+\|\mathrm{X}\|}\right\}, \quad (\varepsilon \to 0^{+}\,\,\,\Rightarrow\,\,\,\varsigma(\varepsilon)\to 0^+),
$$
then the constant $\mathrm{C}>0$ appearing in up-to-the boundary $W^{2,p}$ \textit{a priori} estimates \eqref{2} also depends on $\varsigma$.
\end{remark}

\subsection{Approximation device in geometric tangential analysis} \label{section3}

\hspace{0.4cm} The main result of this section provides a compactness method that will be used as a key point throughout the whole paper. As a matter of fact, we will show that if our equation is close enough to the homogeneous equation with constant coefficients, then our solution is sufficiently close to a solution of the homogeneous equation with frozen coefficients.  At the core of our techniques is the notion of the recession operator (see Definition \ref{DefAC}). The appropriate way to formalize this intuition is with an approximation lemma.


\begin{lemma}[{\bf Approximation Lemma}] \label{Approx}
	Let $n \le p < \infty$, $0 \le \nu \le 1$ and assume that $(A1)-(A4)$ are in force. Then, for every $\hat{\delta} >0$, $\varphi \in C^0(\partial \mathrm{B}_1(0^{\prime},\nu))$ with $\|\varphi\|_{L^{\infty}(\partial \mathrm{B}_1(0^{\prime},\nu))} \le \mathfrak{c}_1$ and $g \in C^{1,\alpha}(\overline{\mathrm{T}}_2)$ with $0 < \alpha < 1$  and $\|g\|_{C^{1,\alpha}(\overline{\mathrm{T}}_2)} \le \mathfrak{c}_2$ for some $\mathfrak{c}_2 >0$ there exist positive constants $\epsilon =\epsilon(\hat{\delta},n, \mu_0, p, \lambda, \Lambda, \gamma,\mathfrak{c}_1, \mathfrak{c}_2) < 1$ and $\tau_0 = \tau_0(\hat{\delta}, n, \lambda,\Lambda, \mu_0, \mathfrak{c}_1, \mathfrak{c}_2) >0$ such that, if
	$$
	\max\left\{ |F_{\tau}(\mathrm{X},x) - F^{\sharp}(\mathrm{X},x)|, \, \|\psi_{F^{\sharp}}(\cdot,0)\Vert_{L^{p}(\mathrm{B}^{+}_{r})},\,\|f\|_{L^{p}(\mathrm{B}^+_{r})}  \right\} \le \epsilon \quad \textrm{and} \quad \tau \le \tau_0
	$$
	then, any two $L^p$-viscosity solutions $u$ (normalized, i.e., $\|u\|_{L^{\infty}(\mathrm{B}^{+}_r(0^{\prime},\nu))}\le 1$) and $\mathfrak{h}$ of
	$$
	\left\{
	\begin{array}{rclcl}
		F_{\tau}(D^2u,x) &=& f(x)& \mbox{in} & \mathrm{B}^{+}_r(0^{\prime},\nu) \cap \mathbb{R}^{n}_+ \\
		 \mathcal{B}(x,u,Du)&=& g(x) & \mbox{on} &  \mathrm{B}_r(0^{\prime},\nu) \cap \mathrm{T}_r\\
		u(x) &=& \varphi(x) &\mbox{on}& \overline{\partial \mathrm{B}_r(0^{\prime},\nu) \cap \mathbb{R}^n_+}
	\end{array}
	\right.
	$$
	and
	$$
	\left\{
	\begin{array}{rclcl}
		F^{\sharp}(D^2 \mathfrak{h},0) &=& 0& \mbox{in} & \mathrm{B}^{+}_{\frac{7}{8}r}(0^{\prime},\nu) \cap \mathbb{R}^n_+ \\
		 \mathcal{B}(x,\mathfrak{h},D\mathfrak{h}) &=& g(x) & \mbox{on} &  \mathrm{B}_{\frac{7}{8} r}(0^{\prime},\nu) \cap \mathrm{T}_r\\
		\mathfrak{h}(x) &=& u(x) &\mbox{on}& \overline{\partial \mathrm{B}_{\frac{7}{8}r}(0^{\prime}, \nu) \cap \mathbb{R}^n_+}		
	\end{array}
	\right.
	$$
	satisfy
	$$
	\|u-\mathfrak{h}\|_{L^{\infty}(\mathrm{B}^{+}_{\frac{7}{8}r}(0^{\prime},\nu))} \le \hat{\delta}.
	$$
\end{lemma}

\begin{proof}
	We may assume, without loss of generality that $r=1$. We will proceed with a \textit{reductio at absurdum} argument. Thus, let us assume that the claim is not satisfied. Then, there exists $\delta_0>0$ and a sequence of functions $(F_{\tau_j})_{j \in \mathbb{N}}$, $(F^{\sharp}_j)_{j \in \mathbb{N}}$, $(u_{j})_{j \in \mathbb{N}}$,  $(f_j)_{j \in \mathbb{N}}$, $(\varphi_j)_{j \in \mathbb{N}}$, $\{g_j\}_{j \in \mathbb{N}}$ and $(\mathfrak{h}_j)_{j \in \mathbb{N}}$ linked through
	$$
	\left\{
	\begin{array}{rclcl}
		F_{\tau_j}(D^2u_j,x) &=& f_j(x)  & \mbox{in} &  \mathrm{B}_1(0^{\prime},\nu_j) \cap \mathbb{R}^{n}_{+}\\
		\mathcal{B}(x,u_{j},Du_{j})&=& g_j(x) & \mbox{on} &  \mathrm{B}_1(0^{\prime},\nu_j) \cap T_1\\
		u_j(x) &=& \varphi_j(x)  &\mbox{on}& \overline{\partial \mathrm{B}_1(0^{\prime},\nu) \cap \mathbb{R}^n_+}
	\end{array}
	\right.
	$$
	and
	$$
	\left\{
	\begin{array}{rclcl}
		F^{\sharp}_{j}(D^2 \mathfrak{h}_j,0) &=& 0 & \mbox{in} &  \mathrm{B}_{\frac{7}{8}}(0^{\prime},\nu_j) \cap \mathbb{R}^{n}_{+}\\
		 \mathcal{B}(x,\mathfrak{h}_{j},D\mathfrak{h}_{j})&=& g_j(x) & \mbox{on} &\mathrm{B}_{\frac{7}{8}}(0^{\prime},\nu_j) \cap \mathrm{T}_1\\
		\mathfrak{h}_j(x) &=& u_j(x) &\mbox{on}& \overline{\partial \mathrm{B}_{\frac{7}{8}}(0^{\prime}, \nu) \cap \mathbb{R}^n_+}.
	\end{array}
	\right.
	$$
	where $\tau_j$, $\Vert \psi_{F_{\tau_{j}}^{\sharp}}(\cdot,0)\Vert_{L^{p}(\mathrm{B}^{+}_{1})}$, $\|f_j\|_{L^p(\mathrm{B}^+_{1})}$ go to zero as $j \to \infty$ and such that
	\begin{equation} \label{1}
		\|u_j-\mathfrak{h}_j\|_{L^{\infty}(\mathrm{B}_{\frac{7}{8}}(0^{\prime}, \nu) \cap \mathbb{R}^n_+)} > \delta_0.
	\end{equation}
	Here,  $\varphi_j\in C^0(\partial \mathrm{B}_{1}(0^{\prime},\nu_j))$ and $g_j \in C^{1,\alpha}(\overline{\mathrm{T}_2})$ satisfy  $\|\varphi_j\|_{L^{\infty}(\partial \mathrm{B}_{1}(0^{\prime},\nu_j))}\leq \mathfrak{c}_{1}$ and $\|g_j\|_{C^{1,\alpha}(\overline{\mathrm{T}_2})} \le \mathfrak{c}_2$, respectively. From Theorem $\ref{Holder_Est}$, we have for all $0<\rho<1$
	\begin{equation} \label{4.6}
		\|u_j\|_{C^{0, \alpha^{\prime}}(  \overline{\mathrm{B}_{1-\rho}(0^{\prime},\nu_j) \cap \mathbb{R}^n_+}     )}\le \mathrm{C}(n,\lambda,\Lambda,\mathfrak{c}_{1},\mathfrak{c}_{2},\mu_0) \rho^{-\alpha^{\prime}}
	\end{equation}
	for some $\alpha^{\prime} = \alpha^{\prime}(n,\lambda,\Lambda, \mu_0)$ and sufficiently large $j$. Suppose that there exists a number $\nu_{\infty}$ and a subsequence $\{\nu_{j_k}\}$ such that $\nu_{j_k} \to \nu_{\infty}$ as $k \to +\infty$. We can assume that such a subsequence is monotone. If $v_{j_k}$ is decreasing, we can check that
	$$
	\mathrm{B}_1(0^{\prime},\nu_{\infty}) \cap \mathbb{R}^n_+ \subset \mathrm{B}_1(0^{\prime},\nu_{j_k}) \cap \mathbb{R}^n_+
	$$
	for any $k$. Thus, we observe that
	\begin{equation} \label{4.7}
		\|u_{j_k}\|_{C^{0, \alpha^{\prime}} ( \overline{\mathrm{B}_{15/16}(0^{\prime},\nu_{\infty}) \cap \mathbb{R}^n_+})}\le C^0(\mathfrak{c}_1,\mathfrak{c}_2,n,\lambda,\Lambda,\mu_0)
	\end{equation}
	by using \eqref{4.6}. On the other hand, if $\nu_{j_k}$ is increasing, there exists a number $k_0$ such that
	$$
	\mathrm{B}_{31/32}(0^{\prime}, \nu_{k_j}) \cap \mathbb{R}^n_+ \supset  \mathrm{B}_{15/16}(0^{\prime},\nu_{\infty}) \cap \mathbb{R}^n_+ \quad \textrm{for} \,\,\,\, k \ge k_0.
	$$
	Then we can deduce once again \eqref{4.7} for some proper subsequence $u_{j_k}$. Thus, according to Arzel\`{a}-Ascoli's compactness criterium, there exist functions $u_{\infty} \in C^{0,\alpha}(\overline{\mathrm{B}_{15/16}(0^{\prime},\nu_{\infty}) \cap \mathbb{R}^n_+})$, $g_{\infty} \in C^0(\partial \mathrm{B}^+_{1})$ and subsequences such that $u_{j_k} \to u_{\infty}$  in $C^{0,\alpha^{\prime}}(\mathrm{B}^+_{1})$ and $u_{\infty}= g_{\infty}$ on $\mathrm{B}_{15/16}(0^{\prime},\nu_{\infty}) \cap \mathrm{T}_1$.
	
	Since the functions $F^{\sharp}_{j}(\cdot, 0) \to F^{\sharp}_{\infty}(\cdot, 0)$  uniformly in compact sets of $ \textit{Sym}(n)$ and for every $\varphi \in C^2(\overline{\mathrm{B}^+_2})$,
	\begin{eqnarray*}
		|F_{\tau_{j_k}}(D^2 \varphi, x) - f_{j_k}(x) - F^{\sharp}_{\infty}(D^2 \varphi, 0)| &\le& |F_{\tau_{j_k}}(D^2 \varphi, x) - F^{\sharp}_{j_k}(D^2 \varphi, x)|+|f_{j_{k}}|+ \\
		&+& |F^{\sharp}_{j_k}(D^2 \varphi, x) - F^{\sharp}_{j_{k}}(D^2 \varphi, 0)|+\\
		&+&|F^{\sharp}_{j_k}(D^2 \varphi, 0)-F^{\sharp}_{\infty}(D^2 \varphi, 0)| \\
		&\le& |F_{\tau_{j_k}}(D^2 \varphi, x) - F^{\sharp}_{j_k}(D^2 \varphi, x)|+|f_{j_{k}}|+\\
		&+&\psi_{F_{\tau_{j_{k}}}^{\sharp}}(x,0)(1+|D^{2}\varphi|)
	\end{eqnarray*}
	then
	$$
	\lim_{k \to +\infty}  \| F_{\tau_{j_k}}(D^2 \varphi,  x) - f_{j_k}(x) - F^{\sharp}_{\infty}(D^2 \varphi, 0)  \|_{L^p(\mathrm{B}_r(x_0))} =0,
	$$
	for any ball $\mathrm{B}_r(x_0) \subset \mathrm{B}_{15/16}(0,\nu_{\infty}) \cap \mathbb{R}^n_+$. Therefore, the Stability Lemma \ref{Est} ensures that $u_{\infty}$ satisfies
	$$
		\left\{
		\begin{array}{rclcl}
			F^{\sharp}_{\infty}(D^2 u_{\infty},0) &=& 0  & \mbox{in} &  \mathrm{B}_{15/16}(0^{\prime},\nu_{\infty}) \cap \mathbb{R}^n_+ \\
			 \mathcal{B}(x,u_{\infty},Du_{\infty})&=& g_{\infty}(x)  & \mbox{on} & \mathrm{B}_{15/16}(0^{\prime},\nu_{\infty}) \cap T_1
		\end{array}
		\right.
$$
	in the viscosity sense. Now, we consider $w_{j_k} \colon= u_{\infty} - \mathfrak{h}_{j_k}$ for each $k$. Then, by \cite[Lemma 4.3]{BJ}, $w_{j_k}$ satisfies
$$
		\left\{
		\begin{array}{rclcl}
			w_{j_k} &\in& \mathcal{S}(\frac{\lambda}{n}, \Lambda, 0 )    & \mbox{in} &  \mathrm{B}_{7/8}(0^{\prime},\nu_{\infty}) \cap \mathbb{R}^n_+ \\
			\ \mathcal{B}(x,w_{j_{k}},Dw_{j_{k}}) &=& g_{\infty}(x)-g_{j_k}(x)  & \mbox{on} & \mathrm{B}_{7/8}(0^{\prime},\nu_{\infty}) \cap T_1\\
			w_{j_k}(x) &=& u_{\infty}(x) - u_{j_k}(x) &\mbox{on}& \overline{\partial \mathrm{B}_{7/8}(0^{\prime},\nu_{\infty}) \cap \mathbb{R}^n_+}.
		\end{array}
		\right.
$$
	Now, using Lemma \ref{ABP-fullversion}, we observe that
	\begin{eqnarray*}
		\|w_{j_k}\|_{L^{\infty}(\mathrm{B}_{7/8}(0^{\prime},\nu_{\infty}) \cap \mathbb{R}^n_+)} &\le& \|u_{\infty} - u_{j_k}\|_{L^{\infty}(\partial \mathrm{B}_{7/8}(0^{\prime},\nu_{\infty}))} +\\
		&+& \mathrm{C}(n,\lambda, \Lambda, \mu_0) \|g_{\infty} - g_{j_k}\|_{L^{\infty}(\mathrm{B}_{7/8}(0^{\prime},\nu_{\infty})\cap \mathrm{T}_1)} \to 0 \quad \text{as}\,\,\, k \to +\infty.
	\end{eqnarray*}
This is, $w_{j_k}$ converges uniformly to zero. This implies that $\mathfrak{h}_{j_k}$ converges uniformly to $u_{\infty}$ in $\overline{\mathrm{B}_{7/8}(0^{\prime}, \nu_{\infty}) \cap \mathbb{R}^n_+}$, which contradicts \eqref{1} for $k \gg 1$.
\end{proof}


\section{$W^{2, p}$ estimates: Proof of Theorem \ref{T1}}\label{Sec4}

\hspace{0.4cm}Before starting the proof of Theorem \ref{T1} we must introduce some useful terminologies: let $$
   \mathcal{Q}^d_r(x_0) \defeq \left(x_0-\frac{r}{2}, x_0 + \frac{r}{2}\right)^d
$$
be the $d-$dimensional cube of side-length $r$ and center $x_0$. In the case $x_0=0$, we will write $\mathcal{Q}^d_r$. Furthermore, if $d=n$ we will write $\mathcal{Q}_r(x_0)$.

We will need later on the Calder\'{o}n-Zygmund decomposition Lemma for cubes: Let $\mathcal{Q}_{1}^{n-1}\times (0,1)$ unit cube. We split it into $2^{n}$ cubes of half side. We do the same splitting procedure with each one of these $2^{n}$ cubes, and we iterate such a process. Each cube in such a procedure is called a dyadic cube. Given two $\mathcal{Q}, \tilde{\mathcal{Q}}\neq \mathcal{Q}_{1}^{n-1}\times(0,1)$ dyadic cubes, we say that $\mathcal{Q}$ is a predecessor cube of $\tilde{\mathcal{Q}}$ if $\mathcal{Q}$ is one of the $2^{n}$ cubes obtained in the partition of $\tilde{\mathcal{Q}}$. In the sequel, we may enunciate a key result whose proof can be found in Caffarelli-Cabr\'{e}'s Book \cite[Lemma 4.2]{CC}.

\begin{lemma}[{\bf Calder\'{o}n-Zygmund cube decomposition}]\label{Cal-Zyg}
 Let $\mathcal{A}\subset \mathcal{B}\subset \mathcal{Q}_{1}^{n-1}\times (0,1)$ be measurable sets and $\delta\in(0,1)$ such that
 \begin{enumerate}
 \item[({\bf a})] $\Leb(\mathcal{A})\leq \delta$;
 \item[({\bf b})] If $\mathcal{Q}$ is a dyadic cube such that $\Leb(\mathcal{A}\cap \mathcal{Q})>\delta \Leb(\mathcal{Q})$, then $\tilde{\mathcal{Q}}\subset \mathcal{B}$.
 \end{enumerate}
Then, $\Leb(\mathcal{A})\leq \delta \Leb(\mathcal{B})$.
\end{lemma}

Remember that the Hardy-Littlewood maximal operator is defined as follows for $f \in L^1_{\textrm{loc}}(\mathbb\R^n)$:
$$
	\mathcal{M}(f)(x) \defeq \sup_{\rho>0} \intav{\mathrm{B}_{\rho}(x)} |f(y)| dy.
$$
We say that $\mathrm{P}_{\mathrm{M}}$ is a paraboloid with \textit{opening} $\mathrm{M}>0$ whenever $\mathrm{P}_{\mathrm{M}}(x)= \pm \frac{\mathrm{M}}{2} |x|^2 + p_1\cdot x + p_0$. Observe that such a paraboloid is a convex function in the case of ``plus'' sign and a concave function otherwise. Now, for $u \in C^0(\Omega)$, $\Omega^{\prime} \subset \overline{\Omega}$ and $\mathrm{M} >0$ we define
$$
	\underline{\mathcal{G}}_{\mathrm{M}}(u,\Omega^{\prime}) \defeq \left\{x_0 \in \Omega^{\prime} ; \exists \, \mathrm{P}_{\mathrm{M}} \,\,\, \textrm{concave parabolid s. t.} \,\,\, \mathrm{P}_{\mathrm{M}}(x_0)=u(x_0), \,\, \mathrm{P}_{\mathrm{M}}(x) \le u(x)\,\, \forall \, x \in \Omega^{\prime}\right\}
$$
and
$$
\underline{\mathcal{A}}_{\mathrm{M}}(u,\Omega^{\prime}) \defeq \Omega^{\prime} \setminus \underline{\mathcal{G}}_{\mathrm{M}}(u,\Omega^{\prime}).
$$

Similarly, by using convex paraboloid we define $\overline{\mathcal{G}}_{\mathrm{M}}(u,\Omega^{\prime})$ and $\overline{\mathcal{A}}_{\mathrm{M}}(u,\Omega^{\prime})$ and the sets
$$
	\mathcal{G}_{\mathrm{M}}(u,\Omega^{\prime}) \defeq  \underline{\mathcal{G}}_{\mathrm{M}}(u,\Omega^{\prime}) \cap \overline{\mathcal{G}}_{\mathrm{M}}(u,\Omega^{\prime})\,\,\,\text{and}\,\,\,\mathcal{A}_{\mathrm{M}}(u,\Omega^{\prime}) \defeq \underline{\mathcal{A}}_{\mathrm{M}}(u,\Omega^{\prime}) \cap \overline{\mathcal{A}}_{\mathrm{M}}(u,\Omega^{\prime}).
$$
In addition, we define:
$$
\overline{\Theta}(u,\Omega^{\prime}, x) \defeq \inf\left\{\mathrm{M}>0 ; \, x \in \overline{\mathcal{G}}_{\mathrm{M}}(u,\Omega^{\prime})\right\}.
$$
We also can define $\underline{\Theta}(u,\Omega^{\prime}, x)$. Finally, we define:
$$
\Theta(u,\Omega^{\prime}, x) \defeq \sup\left\{\underline{\Theta}(u,\Omega^{\prime}, x), \overline{\Theta}(u,\Omega^{\prime}, x)\right\}.
$$

\bigskip

The first step towards $W^{2,p}$-estimates up to the boundary are estimates for paraboloids on the boundary. Hence, in this section, we will prove an appropriated power decay on the boundary for $\Leb(\mathcal{G}_{t}(u,\Omega))$. For this end, we will need the following standard technical result (cf. \cite{CC}):
\begin{proposition}\label{P1}
Let $0\leq h: \Omega \to \R$ be a measurable function, $\mu_{h}(t) \defeq \Leb(\{x \in \Omega: h(x) \ge t\})$
its distribution function and $\eta>0$, $\mathrm{M} >1$ constants. Then,
$$
	h \in L^p(\Omega) \Longleftrightarrow  \sum_{j=1}^{\infty} \mathrm{M}^{pj} \mu_h(\eta \mathrm{M}^j)< \infty
$$
for every $p \in (0,+\infty)$. Particularly, there exists a constant $\mathrm{C}=\mathrm{C}(n,\eta, \mathrm{M})$ such that
$$
	\displaystyle \mathrm{C}^{-1}.\left(\sum_{j=1}^{\infty} \mathrm{M}^{pj}.\mu_h(\eta \mathrm{M}^j)\right)^{\frac{1}{p}} \le \|h\|_{L^p(\Omega)} \le \mathrm{C} \left(\Leb(\Omega) + \sum_{j=1}^{\infty} \mathrm{M}^{pj} \mu_h(\eta \mathrm{M}^j)\right)^{\frac{1}{p}}.
$$
\end{proposition}

Once we have such estimates for paraboloids on the boundary, the ideas will follow as the ones of Caffarelli's original proof for the interior estimates (cf. \cite{CC}). In effect, consider the distribution function of $\Theta$, i,e., $\mu_{\Theta, \Omega^{\prime}}(t) \defeq \Leb\left(\left\{x \in \Omega^{\prime} : \, \Theta(x) > t\right\}\right)$. Then, one check that $\mu_{\Theta, \Omega^{\prime}}(t) = \Leb(\mathcal{A}_t(u,\Omega^{\prime}))$. Moreover, as an application of Proposition \ref{P1} we have
$$
	\Theta(u,\Omega^{\prime}, \cdot) \in L^p(\Omega^{\prime}) \Longleftrightarrow \sum_{j=1}^{\infty} \mathrm{M}^{pj} \Leb\left(\mathcal{A}_{\eta \mathrm{M}^j(u,\Omega^{\prime})}\right) < \infty.
$$
Finally, from \cite[Proposition 1.1]{CC} we conclude
$$
   \|D^2 u\|_{L^p(\Omega^{\prime})} \le \mathrm{C}(\eta,\mathrm{M}, p) \left(\Leb(\Omega) + \sum_{j=1}^{\infty} \mathrm{M}^{pj} \Leb\left(\mathcal{A}_{\eta \mathrm{M}^j(u,\Omega^{\prime})}\right)\right).
$$
Therefore, to obtain the desired $W^{2,p}$-estimates in $\Omega^{\prime}$ it is enough to prove the corresponding summability for $\displaystyle \sum_{j=1}^{\infty} \mathrm{M}^{pj} \Leb\left(\mathcal{A}_{\eta \mathrm{M}^j(u,\Omega^{\prime})}\right)$.

We gather only a few elements involved in the proof of Theorem \ref{T1}, despite the fact that such results are well-known in the literature. We observe that such an \textit{a priori} estimate is independent of further assumptions on the operator $F$, and it follows merely from uniform ellipticity and the integrability of the source term.  Thus, the proof is omitted in what follows. See \cite[Lemma 7.8]{CC} and \cite[Lemma 2.7]{Winter} for details.

\begin{proposition}[{\bf Power Decay on the boundary}]\label{Prop2.12}
Let $u\in \mathcal{S}(\lambda,\Lambda,f)$ in $\mathrm{B}^{+}_{12\sqrt{n}}\subset \Omega\subset\mathbb{R}^{n}_{+}$, $u\in C^0(\Omega)$ and $\Vert u\Vert_{L^{\infty}(\Omega)}\leq 1$. Then, there exist universal constants $\mathrm{C}>0$ and $\mu>0$ such that $\Vert f\Vert_{L^{n}(\mathrm{B}^{+}_{12\sqrt{n}})}\le 1$ implies
$$
	\Leb\left(\mathcal{A}_t(u,\Omega) \cap \left(\mathcal{Q}^{n-1}_1 \times (0,1) + x_0\right)\right) \le \mathrm{C}.t^{-\mu}
$$
for any $x_0 \in \mathrm{B}_{9\sqrt{n}} \cap \overline{\mathbb{R}^n_+}$ and $t>1$.
\end{proposition}

Proposition \ref{Prop2.12} yields key information on the measure of $\mathcal{G}_{\mathrm{M}}(u,\Omega) \cap  \left((\mathcal{Q}^{n-1}_1 \times (0,1)) + x_0\right)$, provided
$$
\mathcal{G}_1(u,\Omega) \cap \left((\mathcal{Q}^{n-1}_2 \times (0,2)) + x_0\right)\not= \emptyset.
$$

The proof of Theorem \ref{T1} will be split into two main steps: Firstly, we will investigate equations governed by operators without dependence on the function or gradient entry. In the sequel, we reduce the analysis for operators with full dependence. The first step towards such estimates is the following Proposition.

\begin{proposition}\label{T-flat}
	Let $u$ be a normalized viscosity solution of
	\begin{equation*} \label{mens}
		\left\{
		\begin{array}{rclcl}
			F(D^2u, x) &=& f(x) & \mbox{in} & \mathrm{B}^+_1,\\
			\mathcal{B}(x,u,Du)&=& g(x) & \mbox{on} & \mathrm{T}_1,
		\end{array}
		\right.
	\end{equation*}
	where $\beta,\gamma, g \in C^{1,\alpha}(\mathrm{T}_1)$ with $\beta \cdot \overrightarrow{\textbf{n}} \ge \mu_0$ for some $\mu_0 >0$, $\gamma \le 0$ and $f \in L^p(\mathrm{B}^+_1)\cap C^{0}(\mathrm{B}^{+}_{1})$, for $n\le  p < \infty$. Further, assume that assumptions (A1)-(A4) are in force. Then, there exist positive constants $\psi_{0}$ and $r_0$ depending on $n$, $\lambda$, $\Lambda$, $p$, $\mu_0$, $\alpha$, $\|\beta\|_{C^{1,\alpha}(\overline{\mathrm{T}_{1}})}$ and $\|\gamma\|_{C^{1,\alpha}(\overline{\mathrm{T}_{1}})}$ such that if
	$$
	\left(\intav{\mathrm{B}_r(x_0) \cap \mathrm{B}^+_1} \psi_{F^{\sharp}}(x, x_0)^p dx\right)^{\frac{1}{p}} \le \psi_0
	$$
	for any $x_0 \in \mathrm{B}^+_1$ and $r \in (0, r_0)$, then $u \in W^{2, p}\left(\overline{\mathrm{B}^+_{\frac{1}{2}}}\right)$ and
$$
		\|u\|_{W^{2, p}\left(\mathrm{B}^+_{\frac{1}{2}}\right)} \le \mathrm{C} \cdot \left( \|u\|_{L^{\infty}(\mathrm{B}^+_1)} + \|f\|_{L^p(\mathrm{B}^+_1)}+\Vert g\Vert_{C^{1,\alpha}(\overline{\mathrm{T}_{1}})}\right),
$$
	where $\mathrm{C}=\mathrm{C}(n,\lambda,\Lambda,\mu_0,p,
	\Vert \gamma\Vert_{C^{1,\alpha}(\overline{\mathrm{T}_{1}})}, \alpha,r_{0})>0$.
\end{proposition}

In order to prove Proposition \ref{T-flat} we need a number of auxiliary results:

\begin{proposition}\label{Prop.4.6}
Assume that assumptions (A1)-(A4) there hold. Let $\mathrm{B}^{+}_{14 \sqrt{n}} \subset \Omega \subset \mathbb{R}^n_+$ and $u$ be a viscosity solution of
$$
\left\{
\begin{array}{rclcl}
 F_{\tau}(D^2u,x) &=& f(x) & \mbox{in} & \mathrm{B}^+_{14\sqrt{n}},\\
 \mathcal{B}(x,u,Du)&=& g(x) & \mbox{on} & \mathrm{T}_{14 \sqrt{n}}.
\end{array}
\right.
$$
Assume further that $\max \left\{ \|f\|_{L^n(\mathrm{B}^+_{14\sqrt{n}})}, \,\,\tau \right\} \le \epsilon$. Finally, suppose for some $\tilde{x}_0 \in \mathrm{B}_{9\sqrt{n}} \cap \{x_n\ge0\}$ the following
$$
\mathcal{G}_1(u,\Omega) \cap \left((\mathcal{Q}^{n-1}_2 \times (0,2)) + \tilde{x}_0\right) \not= \emptyset.
$$
Then,
$$
	\Leb\left(\mathcal{G}_{\mathrm{M}}(u;\Omega) \cap  \left((\mathcal{Q}^{n-1}_1 \times (0,1)) + x_0\right)\right) \ge 1-\epsilon_0,
$$
where $x_0 \in \mathrm{B}_{9 \sqrt{n}} \cap \{x_n \ge 0\}$, $\mathrm{M}>1$ depends only on $n$, $\lambda$, $\Lambda$, $\mu_0$, $\alpha$, $\|\gamma\|_{C^{1,\alpha}(\overline{\mathrm{T}}_{14 \sqrt{n}})}$, $\mathrm{C}_{1}$ (from Assumption A4) and $\|g\|_{C^{1,\alpha}(\mathrm{T}_{14 \sqrt{n}})}$, and $\epsilon_0 \in (0,1)$.
\end{proposition}

\begin{proof}
Consider $x_1 \in \mathcal{G}_1(u,\Omega) \cap  \left(\mathcal{Q}^{n-1}_2 \times (0,2) + \tilde{x}_0\right) $. Then, there exist paraboloids with an opening $t=1$ touching $u$ at $x_1$ from above and below, this is,
$$
-\frac{1}{2}|x-x_1|^2 \le u(x)-\ell(x) \le \frac{1}{2}|x-x_1|^2
$$
for $x \in \Omega$ and an affine function $\ell$.

Now, we set
$$
v(x) = \frac{u(x)-\ell(x)}{\mathrm{C}_{\ast}},
$$
 where $\mathrm{C}_{\ast}>0$ is a dimensional constant selected (large enough) so that $ \|v\|_{L^{\infty}(\mathrm{B}^+_{12\sqrt{n}})} \le 1$ and
$$
-|x|^2 \le v(x) \le |x|^2 \quad \text{in} \quad \Omega \setminus \mathrm{B}^{+}_{12\sqrt{n}}.
$$
Next, we observe that $v$ is a viscosity solution to
$$
\left\{
\begin{array}{rclcl}
	\tilde{F}_{\tau}(D^2v,x) &=& \tilde{f}(x)& \mbox{in} & \mathrm{B}^+_{14\sqrt{n}},\\
	 \mathcal{B}(x,v,Dv)&=&\frac{1}{\mathrm{C}_{\ast}}[g(x)-\beta \cdot D\ell-\gamma\ell]& \mbox{on} & \mathrm{T}_{14 \sqrt{n}}.
\end{array}
\right.
$$
where
$$
	\tilde{F}_{\tau}(\mathrm{X},x) \defeq \frac{1}{\mathrm{C}_{\ast}}F_{\tau}(\mathrm{C}_{\ast} \mathrm{X}, x)  \quad \text{and} \quad \tilde{f}(x) \defeq \frac{1}{\mathrm{C}_{\ast}} f(x)
$$
Now, we consider $\mathfrak{h}$ the function $\epsilon$-close to $u$, coming from Lemma \ref{Approx}, this is, let $\mathfrak{h} \in C^{1,1}(\mathrm{B}^+_{13\sqrt{n}}) \cap C^0(\overline{\mathrm{B}^{+}_{13\sqrt{n}}})$ be solution of
$$
\left\{
\begin{array}{rclcl}
	\tilde{F}^{\sharp}(D^2 \mathfrak{h},0) &=& 0 & \mbox{in} & \mathrm{B}^+_{13\sqrt{n}},\\
	 \mathcal{B}(x,\mathfrak{h}, D\mathfrak{h}) &=&\frac{1}{\mathrm{C}_{\ast}}[g(x)-\beta(x) \cdot D\ell(x)-\gamma(x)\ell(x)]& \mbox{on} & \mathrm{T}_{13\sqrt{n}}.
\end{array}
\right.
$$
such that
$$
\|v-\mathfrak{h}\|_{L^{\infty}(\mathrm{B}^+_{13\sqrt{n}})} \le \hat{\delta}<1
$$
Notice that $\beta \cdot D \ell \in C^{1,\alpha}(\overline{\mathrm{T}}_{14 \sqrt{n}})$ since $\beta \in C^{1,\alpha}(\overline{T}_{14\sqrt{n}})$ and $D \ell$ is a constant vector field. Hence, the A.B.P. Maximum Principle (Lemma \ref{ABP-fullversion}) assures that
\begin{eqnarray*}
\|\mathfrak{h}\|_{L^{\infty}(\mathrm{B}^+_{13\sqrt{n}})} &\le& \| v\|_{L^{\infty}(\partial \mathrm{B}^{+}_{13\sqrt{n}}\setminus \mathrm{T}_{13\sqrt{n}})}+\frac{\mathrm{C}}{\mathrm{C}_{\ast}}\left[\|g\|_{L^{\infty}(\overline{\mathrm{T}}_{13\sqrt{n}})}+|D\ell|\Vert \beta\Vert_{L^{\infty}(\overline{\mathrm{T}}_{14\sqrt{n}})}+\Vert \gamma\ell\Vert_{L^{\infty}(\overline{\mathrm{T}}_{13\sqrt{n}})}\right]\\
&\le& \mathrm{C}(n,\Vert\ell\Vert_{L^{\infty}(\overline{\mathrm{T}}_{14\sqrt{n}})},\Vert \gamma\Vert_{C^{1,\alpha}(\mathrm{T}_{14\sqrt{n}})},\|g\|_{C^{1,\alpha}(\overline{\mathrm{T}}_{14\sqrt{n}})})\\
&\defeq &\widetilde{\mathrm{C}}
\end{eqnarray*}
In consequence, condition (A4) ensures that
$$
\|\mathfrak{h}\|_{C^{1,1}(\overline{\mathrm{B}^+}_{12\sqrt{n}})} \leq \mathrm{C}(\mathrm{C}_{1},\widetilde{\mathrm{C}})\Longrightarrow
\mathcal{A}_{\mathrm{N}}\left(\mathfrak{h},\mathrm{B}^{+}_{12\sqrt{n}}\right) \cap \left((\mathcal{Q}^{n-1}_1 \times (0,1)) + x_0\right)=\emptyset
$$
for some $\mathrm{N}=\mathrm{N}(\mathrm{C}_{1},\widetilde{C})>1$, and we extended $\mathfrak{h} \Big|_{\mathrm{B}^+_{12\sqrt{n}}}$ (continuously) outside $\mathrm{B}^{+}_{12\sqrt{n}}$ such that $\mathfrak{h}=v$ outside $\mathrm{B}^+_{13\sqrt{n}}$ and $\| v-\mathfrak{h}\|_{L^{\infty}(\Omega)} = \|v-\mathfrak{h}\|_{L^{\infty}(\mathrm{B}^+_{12\sqrt{n}})}$. For this reason,
$$
   \|v-\mathfrak{h}\|_{L^{\infty}(\Omega)} \le \overline{\mathrm{C}}(\mathrm{C}_{1},\widetilde{\mathrm{C}})
$$
 and 	
$$
-(\overline{\mathrm{C}}(\mathrm{C}_{1},\widetilde{\mathrm{C}})+|x|^2) \le \mathfrak{h}(x) \le \overline{\mathrm{C}}(\mathrm{C}_{1},\widetilde{\mathrm{C}})+|x|^2 \quad \text{in} \quad  \Omega \setminus \mathrm{B}^+_{12\sqrt{n}}.
$$
 Therefore, there exists $\mathrm{M}_0 = \mathrm{M}_0(\mathrm{C}_{1},\tilde{\mathrm{C}})\ge \mathrm{N}>1$, for which
 $$
 \mathcal{A}_{\mathrm{M}_0}(\mathfrak{h},\Omega) \cap \left((\mathcal{Q}^{n-1}_1 \times (0,1)) + x_0\right)=\emptyset.
 $$
Summarizing,
\begin{equation} \label{Sub}
	\left( \mathcal{Q}^{n-1}_1 \times (0,1) + x_0 \right) \subset \mathcal{G}_{\mathrm{M}_0}(\mathfrak{h},\Omega).
\end{equation}
Now, we define
$$
w(x) \defeq \frac{1}{2\mathrm{C}\hat{\delta}}(v-\mathfrak{h})(x).
$$

Therefore, $w$ fulfills the assumptions of Proposition \ref{Prop2.12}. Thus, we obtain for $t>1$ the following
$$
\Leb\left(\mathcal{A}_t(w,\Omega) \cap \left((\mathcal{Q}^{n-1}_1 \times (0,1)) + x_0\right)\right) \le \mathrm{C}t^{-\mu} \quad (\text{for a} \,\,\,\mu \,\,\text{universal}).
$$
By using $\mathcal{A}_{2\mathrm{M}_0}(u) \subset \mathcal{A}_{\mathrm{M}_0}(w) \cup \mathcal{A}_{\mathrm{M}_0}(\mathfrak{h})$ and \eqref{Sub} we conclude that
$$
\Leb\left(\mathcal{G}_{2\mathrm{M}_0}(v-\mathfrak{h},\Omega) \cap \left((\mathcal{Q}^{n-1}_1 \times (0,1)) + x_0\right)\right) \ge 1-\mathrm{C} \epsilon^{-\mu}.
$$
Finally, we conclude that
$$
\Leb\left(\mathcal{G}_{2\mathrm{M}_0}(v,\Omega) \cap \left((\mathcal{Q}^{n-1}_1 \times (0,1)) + x_0\right)\right) \ge 1-\mathrm{C} \epsilon^{-\mu}.
$$
 The proof is completed by choosing $\epsilon\ll 1$ in an appropriate manner and by setting $\mathrm{M} \equiv 2\mathrm{M}_0$.
\end{proof}

\begin{lemma} \label{lemma4.8}
Given $\epsilon_0  \in (0, 1)$ and let $u$ be a normalized viscosity solution to
$$
\left\{
\begin{array}{rclcl}
 F_{\tau}(D^2u,x) &=& f(x) & \mbox{in} & \mathrm{B}^+_{14\sqrt{n}},\\
 \mathcal{B}(x,u,Du)&=& g(x) & \mbox{on} & \mathrm{T}_{14 \sqrt{n}}.
\end{array}
\right.
$$
Assume that (A1)-(A4) hold true and that $f$ is extended by zero outside $\mathrm{B}^+_{14\sqrt{n}}$. For $x \in \mathrm{B}_{14 \sqrt{n}}$, let
$$
	\max \left\{\tau, \|f\|_{L^n(\mathrm{B}_{14 \sqrt{n}})} \right\} \le \epsilon
$$
for some $\epsilon >0$ depending only on $n, \epsilon_0, \lambda, \Lambda, \mu_0, \alpha$. Then, for $k \in \mathbb{N}\setminus \{0\}$ we define
\begin{eqnarray*}
	\mathcal{A} & \defeq & \mathcal{A}_{\mathrm{M}^{k+1}}(u, \mathrm{B}^+_{14 \sqrt{n}}) \cap \left(\mathcal{Q}^{n-1}_1 \times (0,1)\right)\\
	\mathcal{B} &\defeq & \left(\mathcal{A}_{\mathrm{M}^k}(u, \mathrm{B}^+_{14\sqrt{n}}) \cap \left(Q^{n-1}_1 \times (0,1)\right)\right)\cup \left\{x \in \mathcal{Q}^{n-1}_1 \times (0,1); \mathcal{M}(f^n) \ge (\mathrm{C}_0\mathrm{M}^k)^n \right\},
\end{eqnarray*}
where $\mathrm{M} = \mathrm{M}(n, \mathrm{C}_0)>1$. Then,
$$
\Leb(\mathcal{A}) \le \epsilon_0(n, \epsilon, \lambda, \Lambda)\Leb(\mathcal{B}).
$$
\end{lemma}

\begin{proof}
We will use Lemma \ref{Cal-Zyg}. Observe that $\mathcal{A} \subset \mathcal{B} \subset \left(\mathcal{Q}^{n-1}_1 \times (0,1)\right)$ and of the Proposition \ref{Prop.4.6} we conclude that $\Leb(\mathcal{A}) \le \delta <1$ provided we choice $\delta=\epsilon_0$. Thus, it remains to show the following implication: for dyadic cubes $\mathcal{Q}$
$$
\Leb\left(\mathcal{A} \cap \mathcal{Q}\right) > \epsilon_0 \Leb(\mathcal{Q}) \quad \Rightarrow \quad  \tilde{\mathcal{Q}} \subset \mathcal{B}.
$$
For that purpose, assume that for some $i \ge 1$, $\mathcal{Q}= \left(\mathcal{Q}^{n-1}_{\frac{1}{2^{i}}} \times \left(0, \frac{1}{2^{i}}\right) \right) + x_0$ is a dyadic cube with predecessor $\tilde{\mathcal{Q}} =  \left(\mathcal{Q}^{n-1}_{\frac{1}{2^{i-1}}} \times \left(0, \frac{1}{2^{i-1}}\right)\right) + \tilde{x}_0$. Now, we assume that $\mathcal{Q}$ satisfies
\begin{equation} \label{(14)}
	\Leb\left(\mathcal{A} \cap \mathcal{Q}\right)= \Leb\left(\mathcal{A}_{\mathrm{M}^{k+1}}(u, \mathrm{B}^+_{14\sqrt{n}}) \cap \mathcal{Q}\right) > \epsilon_0 \Leb(\mathcal{Q}),
\end{equation}
however the inclusion $\tilde{\mathcal{Q}} \subseteq \mathcal{B}$ does not hold. In consequence, there must be $x_1 \in \tilde{\mathcal{Q}}\setminus \mathcal{B}$, i.e.,
\begin{equation}\label{(15)}
	x_1 \in \tilde{\mathcal{Q}} \cap \mathcal{G}_{\mathrm{M}^k}(u, \mathrm{B}^+_{14\sqrt{n}})  \quad \text{and} \quad \mathcal{M}(f^n)(x_1) < (\mathrm{C}_0 \mathrm{M}^k)^n.
\end{equation}
In the sequel, we must split the analysis into two cases:
\begin{enumerate}
	\item[] \textbf{Case 1:} \, If $|x_0-(x^{\prime}_0, 0)| < \frac{1}{2^{i-3}}\sqrt{n}$.

	In this case, we define: $\mathrm{T}(y) \defeq (x^{\prime}_0,0) + \frac{1}{2^i}.y$ and $\tilde{u}: \tilde{\Omega} \rightarrow \mathbb{R}$, where $\tilde{\Omega} = \mathrm{T}^{-1}(\Omega)$, given by $\tilde{u}(y) \defeq \frac{2^{2i}}{\mathrm{M}^k}u(\mathrm{T}(y))$. Now, observe that $\mathcal{Q} \subset \left(\mathcal{Q}^{n-1}_1 \times (0,1)\right)$ implies that $\mathrm{B}^+_{14\sqrt{n}/2^i}(x^{\prime}_0,0) \subset \mathrm{B}^+_{14\sqrt{n}}$. Furthermore, such a $\tilde{u}$ is a viscosity solution to
	$$
	\left\{
	\begin{array}{rclcl}
		\tilde{F}_{\tau}(D^2\tilde{u},y) &=& \tilde{f}(x) & \mbox{in} & \mathrm{B}^+_{14\sqrt{n}},\\
		\tilde{\mathcal{B}}(y,\tilde{u},D\tilde{u}) &=&\tilde{g}(x) & \mbox{on} & \mathrm{T}_{14\sqrt{n}} ,
	\end{array}
	\right.
	$$
	where
$$
\left\{
	\begin{array}{rcl}
		\tilde{F}_{\tau}(\mathrm{X},y)& \defeq& \frac{\tau}{\mathrm{M}^k} F\left(\frac{\mathrm{M}^k}{\tau} \mathrm{X}, \mathrm{T}(y)\right),\\
		\tilde{f}(y)&\defeq& \frac{1}{\mathrm{M}^k} f(\mathrm{T}(y)),\\
		\tilde{\mathcal{B}}(y,s, \overrightarrow{v})&\defeq&\tilde{\beta}(y)\cdot \overrightarrow{v}+\tilde{\gamma}(y)s,\\
		\tilde{\beta}(y) &\defeq& \beta(\mathrm{T}(y)), \\
		\tilde{\gamma}(y)&\defeq& \frac{1}{2^{i}}\gamma(\mathrm{T}(y)),\\
		\tilde{g}(y)&\defeq& \frac{2^{i}}{\mathrm{M}^{k}}g(\mathrm{T}(y))
	\end{array}
\right.
$$

	Now notice that $\tilde{F}^{\sharp}$ fulfills $C^{1,1}$-estimates with the same constant as $F^{\sharp}$. Moreover, from \eqref{(15)} we obtain
	$$
	\|\tilde{f}\|_{L^n(\mathrm{B}^+_{14\sqrt{n}})} \le \frac{2^{i}}{\mathrm{M}^{k}} \left(\int_{\mathcal{Q}_{\frac{28\sqrt{n}}{2^i}(x_1)}}|f(x)|^{n}dx\right)^{\frac{1}{n}} \le 2^n \mathrm{C}_0;
	$$
	As a result, $\|\tilde{f}\|_{L^n(\mathrm{B}^+_{14\sqrt{n}})} \le \epsilon$ provided we select $\mathrm{C}_0$ small enough in \eqref{(15)}. In addition, from \eqref{(15)} we conclude that
$$
\mathcal{G}_1(\tilde{u}, \mathrm{T}^{-1}(\mathrm{B}^+_{14\sqrt{n}})) \cap \left(\mathcal{Q}^{n-2}_2 \times (0,2) + 2^i \left(\tilde{x}_0 - (x^{\prime}_0,0)\right) \right) \not= \emptyset.
$$
Furthermore, $|x_0-\tilde{x}_0| \le \frac{1}{2^i} \sqrt{n}$ implies $|2^i(\tilde{x}_0 - (x^{\prime}_0,0))| < 9\sqrt{n}$.

Therefore, we have shown that the assumptions of Proposition \ref{Prop.4.6} are true. Hence, it follows:
	$$
	\Leb\left( \mathcal{G}_{\mathrm{M}}(\tilde{u}, \mathrm{T}^{-1}(\mathrm{B}^+_{14\sqrt{n}})) \cap \left(\left(\mathcal{Q}^{n-1}_1 \times (0,1)\right) + 2^i(x_0-(x^{\prime}_0, 0))\right)\right) \ge 1-\epsilon_0.
	$$
	Thus,
$$\Leb\left(\mathcal{G}_{\mathrm{M}^{k+1}}(u; \mathrm{B}^+_{14\sqrt{n}}) \cap \mathcal{Q}\right) \ge (1-\epsilon_0) \Leb(\mathcal{Q}),
$$
 which contradicts \eqref{(14)}.
	
	\item[]\textbf{Case 2:}\, If $|x_0-(x^{\prime}_0,0)| \ge \frac{1}{2^{i-3}} \sqrt{n}$.

	In this case, we conclude that $\mathrm{B}_{\frac{\sqrt{n}}{2^{i-3}}}\left(x_0+\frac{1}{2^{i+1}}e_n\right) \subset \mathrm{B}^+_{8 \sqrt{n}}$, where $e_n$ is the $n^{\underline{th}}$ unit vector of canonical base. Now, by defining the transformation: $\mathrm{T}(y) \defeq \left(x_0 +\frac{1}{2^{i+1}}e_n\right) + \frac{1}{2^i}y$, we proceed similarly to the first part of the proof. Finally, applying \cite[Lemma 5.2]{PT} instead of Proposition \ref{Prop.4.6} we obtain a contradiction to \eqref{(14)}. This completes the proof of the Lemma.
\end{enumerate}
\end{proof}

In the sequel, we will establish the proof of Proposition \ref{T-flat}.

\begin{proof}[{\bf Proof of Proposition \ref{T-flat}}]
	Fix $x_0 \in \mathrm{B}_{1/2} \cap \{x_n \ge 0\}$. When $x_0 \in \mathrm{T}_{\frac{1}{2}}$, $0 < r < \frac{1-|x_0|}{14\sqrt{n}}$ define:
	$$
	\kappa \defeq \frac{\epsilon r}{\epsilon r^{-1} \|u\|_{L^{\infty}(\mathrm{B}^+_{14r\sqrt{n}}(x_0))}  + \|f\|_{L^{n}(\mathrm{B}^+_{14r\sqrt{n}}(x_0))}+\epsilon r^{-1}\Vert g\Vert_{C^{1,\alpha}(\mathrm{T}_{14r\sqrt{n}}(x_{0}))}}
	$$
	where the constant $\epsilon=\epsilon(n,\epsilon_0,\lambda,\Lambda,p, \mu_0, \alpha, \|\beta\|_{C^{1,\alpha}(\mathrm{T}_1)})$ is as the one in Proposition \ref{Prop.4.6} and $\epsilon_0 \in (0,  1)$ will be determined \textit{a posteriori}. Now, we define: $\tilde{u}(y) \defeq \frac{\kappa}{r^{2}}u(x_0+ry)$. Then, $\tilde{u}$ is a normalized viscosity solution to
	$$
	\left\{
	\begin{array}{rclcl}
		\tilde{F}(D^2 \tilde{u}, y) &=& \tilde{f}(x) & \mbox{in} & \mathrm{B}^+_{14\sqrt{n}},\\
		\tilde{\mathcal{B}}(y,\tilde{u},D\tilde{u})&=&\tilde{g}(x)  & \mbox{on} & \mathrm{T}_{14\sqrt{n}} .
	\end{array}
	\right.
	$$
	where
$$
\left\{
	\begin{array}{rcl}
		\tilde{F}(\mathrm{X}, y) &\defeq& \kappa F\left(\frac{1}{\kappa} \mathrm{X}, ry+x_0\right) \\
		\tilde{f}(y) &\defeq& \kappa f(x_0+ry)\\
		\tilde{\mathcal{B}}(y,s, \overrightarrow{v})&\defeq& \tilde{\beta}(y)\cdot \overrightarrow{v}+\tilde{\gamma}(y)s\\
		\tilde{\beta}(y) &\defeq&  \beta(x_0+ry)\\
		\tilde{\gamma}(y) &\defeq& r\gamma(x_{0}+ry)\\
		\tilde{g}(y)&\defeq& \frac{\kappa}{r}g(x_{0}+ry).
	\end{array}
\right.
$$
	Hence, $\tilde{F}$ fulfills (A1)-(A4). Moreover,
	\begin{eqnarray} \label{(16)}
		\|\tilde{f}\|_{L^n\left(\mathrm{B}^+_{14\sqrt{n}}\right)} &=& \frac{\kappa}{r} \|f\|_{L^n\left(\mathrm{B}^+_{14r\sqrt{n}}\right)} \le \epsilon \,\,\,\,  \textrm{and}  \,\,\, \|\tilde{\beta}\|_{C^{1,\alpha}(\mathrm{T}_{14\sqrt{n}})} \leq 1,
	\end{eqnarray}
	which ensures that the hypotheses of Lemma \ref{lemma4.8} are in force. Now, let $\mathrm{M}>0$ and $\mathrm{C}_0>0$ be as in the Lemma \ref{lemma4.8} and choose $\epsilon_0 \defeq \frac{1}{2\mathrm{M}^p}$. Now, for $k \geq 0$ we define
	\begin{eqnarray*}
		\alpha_k&\defeq & \Leb\left(\mathcal{A}_{\mathrm{M}^k}(\tilde{u}, \mathrm{B}^+_{14\sqrt{n}}) \cap \left(\mathcal{Q}^{n-1}_1 \times (0,1)\right)\right)\\
		\beta_k &\defeq& \Leb\left( \left\{ x \in \left(\mathcal{Q}^{n-1}_1 \times (0,1)\right); \,\, \mathcal{M}(\tilde{f}^n)(x) \ge (\mathrm{C}_0\mathrm{M}^k)^n\right\} \right).
	\end{eqnarray*}
	As a result, Lemma \ref{lemma4.8} implies that $\alpha_{k+1} \le \epsilon_0 \left(\alpha_k+\beta_k\right)$ and thus
	\begin{equation} \label{(17)}
		\alpha_k \le \epsilon^k_0 + \sum_{i=0}^{k-1} \epsilon^{k-i}_0 \beta_i.
	\end{equation}
	On the other hand, from assumption (A2), we have  $\tilde{f}^n \in L^{\frac{p}{n}}(\mathrm{B}^+_{14\sqrt{n}})$, consequently $\mathcal{M}(\tilde{f}^n) \in L^{\frac{p}{n}}(\mathrm{B}^+_{14\sqrt{n}})$. Thus, by \eqref{(16)}
	$$
	\|\mathcal{M}(\tilde{f}^n)\|_{L^{\frac{n}{p}}(\mathrm{B}^+_{14\sqrt{n}})} \le \mathrm{C}(n,p) \|\tilde{f}^n\|_{L^{\frac{n}{p}}(\mathrm{B}^+_{14\sqrt{n}})} \le \mathrm{C}(n,p) \|\tilde{f}\|^n_{L^p(\mathrm{B}^+_{14\sqrt{n}})} \le \mathrm{C}.
	$$
	Therefore, by Proposition \ref{P1} we obtain
	\begin{equation} \label{(18)}
		\sum_{k=1}^{\infty} \mathrm{M}^{p k} \beta_k \le \mathrm{C}(n,p).
	\end{equation}
	Finally, taking into account the choice of $\epsilon_0$, \eqref{(17)} and \eqref{(18)} we conclude that
	$$
	\sum_{k=1}^{\infty} \mathrm{M}^{pk} \alpha_k \le  \sum_{k=1}^{\infty} 2^{-k} + \left(\sum_{k=0}^{\infty}\mathrm{M}^{pk}\beta_k\right) \cdot \left(\sum_{k=1}^{\infty} \mathrm{M}^{pk} \epsilon^k_0\right) \le \mathrm{C}(n,p),
	$$
	which implies that $\|D^{2} \tilde{u}\|_{L^p(\overline{\mathrm{B}^+_{1/2}})} \le \mathrm{C}(n,p,\mathrm{M})$ and consequently
	\begin{equation*} \label{(19)}
		\|D^2 u\|_{L^p\left(\overline{\mathrm{B}^+_{\frac{r}{2}}(x_0)}\right)} \le \mathrm{C}(n,\lambda,\Lambda,p,r) \left(\|u\|_{L^{\infty}(\mathrm{B}^+_1)}+ \|f\|_{L^p(\mathrm{B}^+_1)}+\Vert g\Vert_{C^{1,\alpha}(\mathrm{T}_{1})}\right).
	\end{equation*}
	
	On the other hand, if $x_0 \in \mathrm{B}^+_{1/2}$, we can use the result of interior estimates (cf. \cite[Theorem 6.1]{PT}, see also \cite{daSR19}). Finally, by combining interior and boundary estimates, we obtain the desired results by using a standard covering argument. This completes the proof of the Proposition.
	
\end{proof}


\begin{corollary}
Let $u$ be a bounded viscosity solution of
$$
\left\{
\begin{array}{rclcl}
 F(D^2u,Du,u, x) &=& f(x) & \mbox{in} & \mathrm{B}^+_1,\\
 \mathcal{B}(x,u,Du)&=& g(x) & \mbox{on} & \mathrm{T}_1,
\end{array}
\right.
$$
where $\beta,\gamma,g \in C^{1,\alpha}(\mathrm{T}_1)$ with $\beta \cdot \overrightarrow{\textbf{n}} \ge \mu_0$ for some $\mu_0 >0$, $\gamma \le 0$. Further, assume that  (A1)-(A4) are in force. Then, there exists a constant $\mathrm{C}>0$ depending on  $n$, $\lambda$, $\Lambda$, $p$, $C_{1}$, such that  $u \in W^{2, p}\left(\mathrm{B}^+_{\frac{1}{2}}\right)$ and
$$
	\|u\|_{W^{2, p}\left(\mathrm{B}^+_{\frac{1}{2}}\right)} \le \mathrm{C} \cdot \left( \|u\|_{L^{\infty}(\mathrm{B}^+_1)} + \|f\|_{L^p(\mathrm{B}^+_1)}+\Vert g\Vert_{C^{1,\alpha}(\overline{\mathrm{T}_{1}})}\right).
$$
\end{corollary}

\begin{proof}
Notice that $u$ is also a viscosity solution of
\begin{equation*}
\left\{
\begin{array}{rclcl}
\tilde{F}(D^2u, x) &=& \tilde{f}(x) & \mbox{in} & \mathrm{B}^+_1,\\
 \mathcal{B}(x,u,Du)&=& g(x) & \mbox{on} & \mathrm{T}_1 .
\end{array}
\right.
\end{equation*}
where $\tilde{F}(\mathrm{X},x)  \defeq  F(\mathrm{X},0,0,x)$ and $\tilde{f}$ is a function satisfying
\begin{eqnarray*}
|\tilde{f}|\leq \sigma|Du|+\xi|u|+|f|.
\end{eqnarray*}
Thus, we can apply Proposition \ref{T-flat} and conclude that
\begin{eqnarray}\label{mens2}
\|u\|_{W^{2, p}\left(\mathrm{B}^+_{\frac{1}{2}}\right)} \le \mathrm{C} \cdot \left( \|u\|_{L^{\infty}(\mathrm{B}^+_1)} + \|\tilde{f}\|_{L^p(\mathrm{B}^+_1)}+\Vert g\Vert_{C^{1,\alpha}(\overline{\mathrm{T}_{1}})}\right).
\end{eqnarray}
Now, using the same reasoning as in \cite[Corollary 4.5]{BJ} we conclude that $u\in W^{1,p}(\mathrm{B}_{1}^{+})$ and
\begin{equation}\label{mens3}
\Vert u\Vert_{W^{1,p}(\mathrm{B}_{1}^{+})}\leq \mathrm{C}\cdot (\Vert u\Vert_{L^{\infty}(\mathrm{B}^{+}_{1})}+\Vert f\Vert_{L^{p}(\mathrm{B}^{+}_{1})}).
\end{equation}
Finally, by combining (\ref{mens2}) and (\ref{mens3}) we complete the desired estimate.
\end{proof}

\begin{corollary}\label{Cor}
Let $u$ be a bounded $L^{p}-$viscosity solution of
$$
\left\{
\begin{array}{rclcl}
F(D^2u, Du, u, x) &=& f(x) & \mbox{in} & \mathrm{B}^+_1,\\
 \mathcal{B}(x,u,Du)&=& g(x) & \mbox{on} & \mathrm{T}_1,
\end{array}
\right.
$$
where $\beta,\gamma, g \in C^{1,\alpha}(\mathrm{T}_1)$ with $\beta \cdot \overrightarrow{\textbf{n}} \ge \mu_0$ for some $\mu_0 >0$, $\gamma \le 0$ and $f \in L^p(\mathrm{B}^+_1)$, for $n \leq p < \infty$. Further, assume that  $F^{\sharp}$ satisfies (A4) and $F$ fulfills the condition $(SC)$. Then, there exist positive constants $\beta_0=\beta_0(n,\lambda,\Lambda,p)$, $r_0 = r_0(n, \lambda, \Lambda, p)$ and $\mathrm{C}=\mathrm{C}(n,\lambda,\Lambda,p,r_0)>0$, such that if
$$
\left(\intav{\mathrm{B}_r(x_0) \cap \mathrm{B}^+_1} \psi_{F^{\sharp}}(x, x_0)^p dx\right)^{\frac{1}{p}} \le \psi_0
$$
for any $x_0 \in \mathrm{B}^+_1$ and $r \in (0, r_0)$, then $u \in W^{2, p}\left(\overline{\mathrm{B}^+_{\frac{1}{2}}}\right)$ and
$$
\|u\|_{W^{2, p}\left(\overline{\mathrm{B}^+_{\frac{1}{2}}}\right)} \le \mathrm{C} \cdot\left( \|u\|_{L^{\infty}(\mathrm{B}^+_1)} +\|f\|_{L^p(\mathrm{B}^+_1)}+\Vert g\Vert_{C^{1,\alpha}(\overline{\mathrm{T}_{1}})}\right).
$$
\end{corollary}

\begin{proof}
It is sufficient to prove the result for equations with no dependence on $Du$ and $u$ (see \cite[Theorem 4.3]{Winter} for details). We will approximate $f$ in $L^p$ by functions $f_j \in C^{\infty}(\overline{\mathrm{B}^+_1}) \cap L^p(\mathrm{B}^+_1)$ such that $f_{j}\to f$ in $L^{p}(\mathrm{B}^{+}_{1})$, and also approximate $g$ by a sequence $(g_{j})$ in $C^{1,\alpha}(\mathrm{T}_{1})$ with the property that $g_{j}\to g$ in $C^{1,\alpha}(\mathrm{T}_{1})$. By Theorems of Uniqueness and Existence (Theorem \ref{Unicidade} and \ref{Existencia}), there exists a sequence of functions $ u_{j}\in C^{0}(\overline{\mathrm{B}^{+}_{1}})$, which they are viscosity solutions of following family of PDEs:
$$
\left\{
\begin{array}{rclcl}
F(D^2u_j, x) &=& f_j(x) & \mbox{in} & \mathrm{B}^+_1,\\
\mathcal{B}(x,u_{j},Du_{j})&=& g_{j}(x) & \mbox{on} & \mathrm{T}_{1}\\
u_{j}(x)&=&u(x) & \mbox{on} & \partial \mathrm{B}^{+}_{1}\setminus \mathrm{T}_{1}.
\end{array}
\right.
$$
Therefore, the assumptions of the Theorem \ref{T-flat} are in force. As a result,
$$
\|u_j\|_{W^{2,p}\left(\overline{\mathrm{B}^{+}_{\frac{1}{2}}}\right)} \le \mathrm{C}(\verb"universal").\left( \|u_j\|_{L^{\infty}(\mathrm{B}^+_1)} + \|f_j\|_{L^p(\mathrm{B}^+_1)}+\Vert g_{j}\Vert_{C^{1,\alpha}(\overline{\mathrm{T}_{1}})}\right),
$$
for a constant $\mathrm{C}>0$. Furthermore, a standard covering argument also yields $u_j \in W^{2,p}(\mathrm{B}^+_1)$. From Lemma $\ref{ABP-fullversion}$, $(u_j)_{j \in \mathbb{N}}$ is uniformly bounded in $W^{2,p}(\overline{\mathrm{B}^+_{\rho}})$ for $\rho  \in (0, 1)$.

Once again, we can apply Lemma \ref{ABP-fullversion} and obtain
$$
\|u_j-u_k\|_{L^{\infty}\left(\mathrm{B}^+_{1}\right)} \le \mathrm{C}(n,\lambda,\Lambda, \mu_0)(\|f_j-f_k\|_{L^p(\mathrm{B}^+_1)}+\Vert g_{j}-g_{k}\Vert_{C^{1,\alpha}(\overline{\mathrm{T}_{1}})}).
$$
Thus, $u_j \to u_{\infty} \quad \textrm{in} \quad C^0(\overline{\mathrm{B}^+_1})$. Moreover, since $(u_j)_{j \in \mathbb{N}}$ is bounded in $W^{2,p}\left(\mathrm{B}^+_{\frac{1}{2}}\right)$  we obtain $u_j \to u_{\infty}$ weakly in $W^{2,p}\left(\mathrm{B}^+_{\frac{1}{2}}\right)$.
Thus,
$$
\|u_{\infty}\|_{W^{2,p}(\overline{\mathrm{B}^{+}_{1/2}})} \le \mathrm{C} \cdot\left( \|u_{\infty}\|_{L^{\infty}(\mathrm{B}^+_1)} + \|f\|_{L^p(\mathrm{B}^+_1)}+\Vert g\Vert_{C^{1,\alpha}(\overline{\mathrm{T}_{1}})}\right).
$$
Finally, stability results (see Lemma \ref{Est}) ensure that $u_{\infty}$ is an $L^{p}-$viscosity solution to
$$
\left\{
\begin{array}{rclcl}
F(D^2u_{\infty}, x) &=& f(x)& \mbox{in} & \mathrm{B}^+_1,\\
 \mathcal{B}(x,u_{\infty},Du_{\infty})&=& g(x) & \mbox{on} & \mathrm{T}_1 \\
u_{\infty}(x)&=& u(x) & \mbox{on} & \partial \mathrm{B}^{+}_{1}\setminus \mathrm{T}_{1}.
\end{array}
\right.
$$
Thus, $w \defeq u_{\infty}-u$ fulfills
$$
\left\{
\begin{array}{rclcl}
w\in S(\frac{\lambda}{n}, \Lambda,0) & \mbox{in} & \mathrm{B}^+_1,\\
\mathcal{B}(x, w, Dw)= 0 & \mbox{on} & \mathrm{T}_1 \\
w= 0 & \mbox{on} & \partial \mathrm{B}^{+}_{1}\setminus \mathrm{T}_{1}.
\end{array}
\right.
$$
which by Lemma $\ref{ABP-fullversion}$ we can conclude that $w=0$ in $\overline{\mathrm{B}^{+}_{1}}\setminus \mathrm{T}_{1}$. By continuity, $w=0$ in $\overline{\mathrm{B}^{+}_{1}}$ which completes the proof.
\end{proof}

\bigskip

Finally, we are now in position to prove the Theorem \ref{T1}.

\begin{proof}[{\bf Proof of Theorem \ref{T1}}]
Based on standard reasoning (cf. \cite{Winter}) it is important to highlight that it is always possible to perform a change of variables in order to flatten the boundary. For this end, consider $x_0 \in \partial \Omega$. Since $\partial \Omega \in C^{2,\alpha}$ there exists a neighborhood of $x_0$, which we will label $\mathcal{V}(x_0)$ and a $C^{2, \alpha}-$diffeomorfism $\Phi: \mathcal{V}(x_0) \to \mathrm{B}_1(0)$ such that $    \Phi(x_0) = 0 \quad \mbox{and} \quad \Phi(\Omega \cap \mathcal{V}(x_0)) = \mathrm{B}^{+}_1$. Now, for $\tilde{\varphi} \in W^{2,p}(\mathrm{B}^+_1)$ we set $\varphi = \tilde{\varphi} \circ \Phi \in W^{2, p}(\mathcal{V}(x_0))$. Thus, we obtain that
$$
  D \varphi = (D \tilde{\varphi} \circ \Phi) D \Phi \quad \text{and}\quad
	D^2 \varphi = D \Phi^t \cdot (D^2 \tilde{\varphi} \circ \Phi) \cdot D \Phi + ((D \tilde{\varphi} \circ \Phi) \partial_{ij} \Phi)_{1 \le i,j \le n}.
$$
Moreover, $\tilde{u} \defeq u \circ \Phi^{-1} \in C^0(\mathrm{B}^+_1)$. Next, we observe that $\tilde{u}$ is an $L^{p}-$viscosity solution to
$$
\left\{
\begin{array}{rclcl}
 \tilde{F}(D^2 \tilde{u}, D \tilde{u}, \tilde{u}, y) &=& \tilde{f}(x) & \mbox{in} & \mathrm{B}^+_1,\\
\tilde{\mathcal{B}} (y, \tilde{u}, D \tilde{u}) & = & \tilde{g}(x) &\mbox{on} & \mathrm{T}_1.
\end{array}
\right.
$$
where for $y=\Phi^{-1}(x)$ we have
$$
\left\{
\begin{array}{rcl}
\tilde{F}\left(D^2 \tilde{\varphi},D \tilde{\varphi}, x\right) & \defeq & F\left(D\Phi^t(y) D^2 \tilde{\varphi} D\Phi(y) + D \tilde{\varphi}D^2 \Phi(y), D \tilde{\varphi}D\Phi(y), \tilde{u}, y\right)\\
\tilde{f}(x) & \defeq & f \circ \Phi^{-1}(x)\\
\tilde{\mathcal{B}}(x,s, \overrightarrow{v})& \defeq &\tilde{\beta}(x)\cdot \overrightarrow{v}+\tilde{\gamma}(x)s\\
\tilde{\beta}(x) & \defeq & (\beta \circ \Phi^{-1}) \cdot (D \Phi \circ \Phi^{-1})^{t}\\
\tilde{\gamma}(x) & \defeq & (\gamma \circ \Phi^{-1}) \cdot (D \Phi \circ \Phi^{-1})^{t}\\
\tilde{g}(x) & \defeq & g \circ \Phi^{-1}
\end{array}
\right.
$$
Furthermore, note that $\tilde{F}(\mathrm{X}, \varsigma, \eta, y) = F\left(D \Phi^t(y) \cdot \mathrm{X} \cdot D \Phi(y) + \varsigma D^2\Phi, \varsigma D\Phi(y), \eta, y\right)$ is a uniformly elliptic operator with ellipticity constants $\lambda \mathrm{C}(\Phi)$, $\Lambda \mathrm{C}(\Phi)$.
Thus,
$$
	\tilde{F}^{\sharp}(\mathrm{X}, \varsigma, \eta, x) = F^{\sharp}\left(D\Phi^t(\Phi^{-1}(x)) \cdot \mathrm{X}\cdot D\Phi(\Phi^{-1}(x)) +  \varsigma D^2 \Phi(\Phi^{-1}(x)), 0,0, \Phi^{-1}(x)\right).
$$
In consequence, we conclude that $\psi_{\tilde{F}^{\sharp}}(x,x_0) \le \mathrm{C}(\Phi) \psi_{F^{\sharp}}(x,x_0)$, ensuring that $\tilde{F}$ falls into the assumptions of Corollary \ref{Cor}. This finishes the proof, as well as establishes the proof of the Theorem \ref{T1}.
\end{proof}


\section{BMO type estimates: Proof of Theorem \ref{BMO}}\label{Section5}

\hspace{0.4cm} This section will be devoted to establishing boundary BMO type estimates. Before proving Theorem \ref{BMO} we will present some key results that play a crucial role in our strategy.

\begin{lemma}[{\bf Approximation Lemma II}]\label{Prop-BMO}
	Let $g \in C^{1,\psi}(\mathrm{T}_1)$ such that $\|g\|_{C^{1,\psi}(\mathrm{T}_1)} \le \mathrm{C}_g$ and $\beta \in C^{1, \psi}(\overline{\mathrm{T}_1})$. Let $F$ and $F^{\infty}$ be $\left(\frac{\lambda}{n}, \Lambda\right)$-elliptic operators. Given $\hat{\delta}_0>0$, there exists $\epsilon_0=\epsilon_0(\hat{\delta}_0,n,\lambda,\Lambda, \mathrm{C}_g)<1$ such that, if
	$$
	\max\left\{\frac{|F(\mathrm{X}, x)-F^{\infty}(\mathrm{X}, x_0)|}{\|X\|+1}, \,\,\psi_{F^{\infty}}(x), \,\,\|f\|_{L^p(\mathrm{B}^+_1)\cap p-\text{BMO}(\mathrm{B}^+_1)}\right\} \le \epsilon_0,
	$$
	then any two $L^{p}-$viscosity solutions $u$ and $v$ of
	$$
	\left\{
	\begin{array}{rclcl}
		F(D^2 u , x) &=& f(x) & \mbox{in} & \mathrm{B}^+_1\\
		\beta \cdot Du &=& g(x) & \mbox{on} & \mathrm{T}_1 .
	\end{array}
	\right.
	\quad \textrm{and} \quad
	\left\{
	\begin{array}{rclcl}
		F^{\infty}(D^2 \mathfrak{h}, x_0) &=& 0 & \mbox{in} & \mathrm{B}^+_1\\
		\beta \cdot D\mathfrak{h} &=& g(x) & \mbox{on} &\mathrm{T}_1.
	\end{array}
	\right.
	$$
	satisfy
$$
\|u-\mathfrak{h}\|_{L^{\infty}\left(\mathrm{B}^+_\frac{7}{8}\right)} \le \hat{\delta}_0.
$$
\end{lemma}

\begin{proof}
The proof includes the same lines as Lemma \ref{Approx} along with minor changes.
\end{proof}

\begin{lemma}[{\bf A quadratic approximation}] \label{BMO4}
	Under the assumptions from Lemma \ref{Prop-BMO}, assume that $F^{\infty}$ satisfies (A4)$^{\star}$. Let $u$ be an $L^p-$viscosity solution to
	$$
	\left\{
	\begin{array}{rclcl}
		F(D^2 u, x) &=& f(x) & \mbox{in} & \mathrm{B}^+_1,\\
		\beta \cdot Du &=& g(x)& \mbox{on} &  \mathrm{T}_1 .
	\end{array}
	\right.
	$$
	Then, there exist universal constants $\mathrm{C}_{\sharp}>0$ and $r_0 \in (0, 1/2)$, as well as a quadratic polynomial $\mathfrak{P}$, with $\|\mathfrak{P}\|_{L^{\infty}(\mathrm{B}^+_1)} \le \mathrm{C}_{\sharp}$ such that
	$$
\displaystyle	\sup_{\mathrm{B}^+_r} |u(x) - \mathfrak{P}(x)| \le r^2 \quad \text{for every} \quad r \le r_0
	$$
\end{lemma}

\begin{proof}
	Firstly, take $\hat{\delta}_0 \in (0,1)$, to be chosen \textit{a posteriori}, and apply Lemma \ref{Prop-BMO} to obtain $\epsilon_0 >0$ and an $L^p-$viscosity solution to
	$$
	\left\{
	\begin{array}{rclcl}
		F^{\infty}(D^2 \mathfrak{h}, x) &=& 0 & \mbox{in} & \mathrm{B}^+_1,\\
		\beta \cdot D\mathfrak{h} &=&g(x)& \mbox{on} &  \mathrm{T}_1.
	\end{array}
	\right.
	$$
	such that
$$\displaystyle \sup_{\mathrm{B}^+_1} |u-\mathfrak{h}| \le \hat{\delta}_0.
$$
Remember that $F^{\infty}$ fulfills assumption (A4)$^{\star}$, then $\mathfrak{h}$ enjoys classical estimates. As a result, its Taylor's expansion of second order, namely $\mathfrak{P}$, is well-defined. Furthermore, we have
	$$
	\sup_{\mathrm{B}^+_{\rho}} |\mathfrak{h}-\mathfrak{P}| \le \mathrm{C}_{\ast} \rho^{2+\psi} \quad \forall \,\, \rho\leq r_0.
	$$
	Now, we select $0<r \ll 1$ (small enough) in such a way that
$$
r \le \min\left\{r_0, \left(\frac{1}{10\mathrm{C}_{\ast}}\right)^{\frac{1}{\psi}}\right\}.
$$
Thus,
$$
\displaystyle \sup_{\mathrm{B}^+_r} |\mathfrak{h}-\mathfrak{P}| \le \frac{1}{10}r^2.
$$
On the other hand, we select $\hat{\delta}_0 \defeq \frac{1}{10}r^2$. As a result, we obtain
$$
	\displaystyle \sup_{\mathrm{B}^+_r} |u- \mathfrak{h}| \le \frac{1}{10}r^2.
$$
Finally, by combining the previous inequalities, we conclude that
	$$
	\sup_{\mathrm{B}^+_r} |u-\mathfrak{P}| \le r^2.
	$$
\end{proof}

As a byproduct of the above result, we can yield a quadratic approximation at the recession level. The proof holds along with straightforward modifications in the Lemma \ref{BMO4}. For this reason, we will omit it here.

\begin{corollary} \label{BMO3}
	Assume the assumptions from Lemma \ref{Prop-BMO}. Then, there exist universal constants $\mathrm{C}_{\sharp}>0$, $\tau_0 >0$ and $r>0$, such that if $u$ is a (normalized) viscosity solution of
	$$
	\left\{
	\begin{array}{rclcl}
		F_{\tau}(D^2 u, x) &=& f(x) & \mbox{in} & \mathrm{B}^+_1,\\
		\beta(x) \cdot Du(x)  &=& g(x) & \mbox{on} & \mathrm{T}_1.
	\end{array}
	\right.
	$$
	with $\max\left\{\tau, \,\|f\|_{p-BMO(\mathrm{B}^+_1)}\right\} \le \tau_0$, there exists a quadratic polynomial $\mathfrak{P}$, with $\|\mathfrak{P}\|_{\infty} \le \mathrm{C}_{\sharp}$ satisfying
	$$
	\sup_{\mathrm{B}^+_r} |u(x)-\mathfrak{P}(x)| \le r^2.
	$$
\end{corollary}

We are already prepared to prove Theorem \ref{BMO} in the face of such above auxiliary results.

\begin{proof}[{\bf Proof of Theorem \ref{BMO}}]
In fact, for $\kappa \in (0,1)$ to be determined \textit{a posteriori}, we define $v(x) \defeq  \kappa u(x)$ such that $v$ is a (normalized) $L^p-$viscosity solution to
	$$
	\left\{
	\begin{array}{rclcl}
		F_{\tau}(D^2 v, x) &=& \tilde{f}(x) & \mbox{in} & \mathrm{B}^+_1,\\
		\beta(x) \cdot Dv(x) &=& \tilde{g}(x) &\mbox{on} & \mathrm{T}_1.
	\end{array}
	\right.
	$$
	where $\tau \defeq \kappa$, $\tilde{f}(x) \defeq \kappa f(x)$ and $\tilde{g}(x) \defeq \kappa g(x)$. Now, we are going to determine $\kappa$. This will be done in such a way that $\max\left\{\tau, \|\tilde{f}\|_{\textrm{p-BMO}(\mathrm{B}^+_1)}\right\} \le \tau_0$, where $\tau_0$ comes from Corollary \ref{BMO3}. Finally, we will establish the result for $v$, which will lead to the Theorem's statement.
	
From now on, we will show that there exists a sequence of second order polynomials $(\mathfrak{P}_k)_{k \in \mathbb{N}}$ given by $\mathfrak{P}_k(x)= \frac{1}{2} x^T \cdot \mathcal{A}_k \cdot x + \mathcal{B}_k \cdot x + \mathcal{C}_k$ such that
	\begin{eqnarray}	
		F^{\sharp}(\mathcal{A}_k, x) & = & (\tilde{f} )_1, \label{proof-uc-eq01} \\
		r^{2(k-1)} |\mathcal{A}_{k} - \mathcal{A}_{k-1}| + r^{k-1} |\mathcal{B}_{k-1} - \mathcal{B}_{k}| + |\mathcal{C}_{k-1}-\mathcal{C}_k| & \le & \mathrm{C}_{\sharp}.r^{2(k-1)},\label{proof-uc-eq05}
	\end{eqnarray}	
	where $\mathrm{C}_{\sharp}>0$ is a universal constant, $0<r\ll 1$ comes from Lemma \ref{BMO4} and
	$$
		\sup_{\mathrm{B}^+_{r^k}} \left|v(x) - \mathfrak{P}_k(x)\right| \le  r^{2k}.
	$$
	The proof will follow by way of an induction process. Let us define $\mathfrak{P}_0$ and $\mathfrak{P}_{-1}$ to be $\mathfrak{P}_{0}(x) = \mathfrak{P}_{-1}(x) = \frac{1}{2} x^T \cdot \mathrm{X}_0 \cdot x$, where $\mathrm{X}_0 \in \textit{Sym}(n)$ fulfills $F^{\sharp}(\mathrm{X}_0, x) = ( \tilde{f})_1$.
	
	The first step of the argument, i.e., the case $k=0$, is naturally verified. Now, suppose we have established the existence of such polynomials for $k=0,1,\ldots, j$. Then, we define the following auxiliary function $v_j: \mathrm{B}^+_1 \rightarrow \mathbb{R}$ given by
	$$
	v_{j}(x) \defeq \frac{(v-\mathfrak{P}_j)(r^j x)}{r^{2j}},
	$$
	where we have by the induction assumption that $v_j$ is a (normalized) $L^p-$viscosity solution to
	$$
	\left\{
	\begin{array}{rclcl}
		F_{j}(D^2 v_j, x) &=&\tilde{f}_j(x) & \mbox{in} & \mathrm{B}^+_1\\
		\tilde{\beta}(x) \cdot Dv_j(x) &=& \tilde{g}_{j}(x)  & \mbox{on} & \mathrm{T}_{1}.
	\end{array}
	\right.
	$$
	where
$$
\left\{
\begin{array}{rcl}
  F_j(\mathrm{X}, x) & \defeq & \tau F \left(\frac{1}{\tau}(\mathrm{X} + \mathcal{A}_j), r^jx)\right) \\
  \tilde{f}_j(x) & \defeq & \tilde{f}(r^j x) \\
  \tilde{g}_{j}(x)  & \defeq & \tilde{g}(r^{j}x)-\frac{1}{r^{j}}\tilde{\beta}(x) \cdot D\mathfrak{P}_j(r^{j}x)\\
  \tilde{\beta}(x) & \defeq & r^{j} \beta(r^{j}x)
\end{array}
\right.	
$$
	with
	$$
	\begin{array}{lll}
		\|f_j\|_{\textrm{p-BMO}(\mathrm{B}^{+}_1)} =   \displaystyle \sup_{0<s\leq 1} \left( \intav{\mathrm{B}^{+}_{s}} \left|f_j(x) - (f_j)_{s}\right|^pdx\right)^{\frac{1}{p}}\\
		\hspace{2.35 cm}= \displaystyle \sup_{0<s\leq 1} \left( \intav{\mathrm{B}^{+}_{sr}} \left|f(z) - (f)_{sr}\right|^pdz\right)^{\frac{1}{p}} \\
		\hspace{2.35 cm} \leq \|\tilde{f}\|_{p-\textrm{BMO}(\mathrm{B}^{+}_1)} \\
		\hspace{2.35 cm}\leq \tau_0.
	\end{array}
	$$
	Furthermore, since $F^{\sharp}(\mathcal{A}_j, x)= ( \tilde{f} )_1$, the PDE with oblique boundary conditions
	$$
	\left\{
	\begin{array}{rclcl}
		F^{\sharp}_{j}(D^2 \mathfrak{h}, x) &=& ( \tilde{f} )_1 & \mbox{in} & \mathrm{B}^+_1\\
		\beta \cdot D \mathfrak{h}(x) & = & \tilde{g}(x) &\mbox{on} &  \mathrm{T}_1
	\end{array}
	\right.
	$$
	satisfies the same (up to the boundary) $C^{2,\psi}$ \textit{a priori} estimates from the problem driven by $F^{\sharp}$, and it is under the assumption of Corollary \ref{BMO3}. Thus, there exists a quadratic polynomial $\tilde{\mathfrak{P}}$ with $\|\tilde{\mathfrak{P}}\|_{\infty} \le \mathrm{C}_{\sharp}$ such that
	\begin{equation}\label{BMO5}
		\sup_{\mathrm{B}^+_r} |v_{j}-\tilde{\mathfrak{P}}| \le r^2.
	\end{equation}
	Rewriting \eqref{BMO5} back to the original domain yields
	$$
	\displaystyle \sup_{\mathrm{B}^{+}_{r^{k+1}}} \left|v(x) - \left[\mathfrak{P}_k(x)+ r^{2k}\tilde{\mathfrak{P}}\left(\frac{x}{r^k}\right)\right]\right|\leq r^{2(k+1)}.
	$$

	Finally, by defining $\mathfrak{P}_{j+1}(x) \defeq  \mathfrak{P}_j(x) + r^{2j} \tilde{\mathfrak{P}}(r^{-j}x)$ we check the $(k+1)^{\underline{th}}$ step of induction. Furthermore, the required conditions \eqref{proof-uc-eq01} and \eqref{proof-uc-eq05} are satisfied. In conclusion, for $\rho>0$, choose an integer $k$  such that $r^{k+1} < \rho \le r^{k}$. Then, using Theorem \ref{T1} to $v_k$, we obtain
	$$
\begin{array}{rcl}
  \displaystyle \sup_{\rho \in (0, 1/2)}\left(\intav{\mathrm{B}^+_{\rho} } |D^2 v(z) - \mathcal{A}_k|^p dz\right)^{\frac{1}{p}} & \le &  \displaystyle \sup_{\rho \in (0, 1/2)} \left(\frac{1}{r^n}.\intav{\mathrm{B}^+_{r^k} } |D^2 v(z) - \mathcal{A}_k|^p dz\right)^{\frac{1}{p}}\\
   & = &  \displaystyle \sup_{\rho \in (0, 1/2)} \left(\frac{1}{r^n}.\int_{\mathrm{B}^+_1 } |D^2 v_k|^p dx\right)^{\frac{1}{p}} \\
   & \le & \mathrm{C}(\verb"universal").
\end{array}
	$$

Now, recall the general inequality:
$$
\displaystyle \intav{\mathrm{B}^{+}_{\rho} } \left|D^2 v - \intav{\mathrm{B}^{+}_{\rho}} D^2 v \ dy\right|^p dx \le   2^{p} \displaystyle  \intav{\mathrm{B}^{+}_{\rho} } |D^2 v - \mathcal{A}_k|^p dx.
$$
In effect, by triangular inequality,
$$
\begin{array}{rcl}
	\displaystyle \left(\intav{\mathrm{B}^{+}_{\rho} } \left|D^2 v - \intav{\mathrm{B}^{+}_{\rho}} D^2 v \ dy\right|^p dx\right)^{\frac{1}{p}} & \le &  \displaystyle  \left(\intav{\mathrm{B}^{+}_{\rho} } |D^2 v - \mathcal{A}_k|^p dx\right)^{\frac{1}{p}}+\left|\intav{\mathrm{B}^{+}_{\rho} } D^2 v \ dy - \mathcal{A}_k\right|\\
	& \le &   \displaystyle  \left(\intav{\mathrm{B}^{+}_{\rho} } |D^2 v - \mathcal{A}_k|^p dx\right)^{\frac{1}{p}}+  \displaystyle  \intav{\mathrm{B}^{+}_{\rho} } |D^2 v - \mathcal{A}_k| dx\\
	& \le & 2 \displaystyle  \left(\intav{\mathrm{B}^{+}_{\rho} } |D^2 v - \mathcal{A}_k|^p dx\right)^{\frac{1}{p}},
\end{array}
$$

Therefore,
$$
\displaystyle \|Dv\|_{p-BMO\left(\overline{\mathrm{B}^{+}_{\frac{1}{2}}}\right)} \defeq \sup_{\rho \in (0, 1/2)}\left(\intav{\mathrm{B}^{+}_{\rho} } \left|D^2 v - \intav{\mathrm{B}^{+}_{\rho}} D^2 v \ dy\right|^p dx\right)^{\frac{1}{p}} \leq \mathrm{C}(\verb"universal"),
$$
thereby finishing the proof of the estimate \eqref{BMO2}.
\end{proof}


\section{Obstacle-type problems: Proof of Theorem \ref{T3}}\label{A1} \label{obst}

\hspace{0.4cm} In this section, we will look at an important class of free boundary problems, namely the obstacle problem with oblique boundary datum \eqref{obss1} for a given obstacle $\phi \in W^{2,p}(\Omega)$ satisfying $\mathcal{B}(x,u, D \phi) \ge g$ a.e. on $\partial \Omega$, and $F$ a uniformly elliptic operator fulfilling $F(\mathcal{O}_{n \times n},\overrightarrow{0},0,x)=0$. The primary purpose will be to develop a $W^{2,p}$-regularity theory for \eqref{obss1} that does not rely on any extra regularity assumptions on nonlinearity $F$.

In the sequel, similarly to \cite[Theorem 3.3]{BLOP18} and \cite[Appendix]{daSV21-1} we will reduce the analysis of the obstacle problem eqrefobss1 in studying a family of penalized problems as follows:
\begin{equation} \label{NN4-4}
		\left\{
		\begin{array}{rclcl}
			F(D^2u_{\varepsilon},Du_{\varepsilon},u_{\varepsilon},x) &=& \mathrm{h}^+(x) \Psi_{\varepsilon}(u_{\varepsilon} - \phi) + f(x) - \mathrm{h}^+(x)& \mbox{in} & \Omega \\
			 \mathcal{B}(x,u_{\varepsilon},Du_{\varepsilon}) &=& g(x)  & \mbox{on} &\partial \Omega,
		\end{array}
		\right.
	\end{equation}
for $\varepsilon \in (0, 1)$, where $\Psi_{\varepsilon}(s)$ is a smooth function such that
$$
	\Psi_{\varepsilon}(s) \equiv 0 \quad \textrm{if} \quad s \le 0; \quad \Psi_{\varepsilon}(s) \equiv 1 \quad \textrm{if} \quad s \ge \varepsilon,
$$
$$
	0 \le \Psi_{\varepsilon}(s) \le 1 \quad \textrm{for any} \quad s \in \mathbb{R}.
$$
and
$$
\mathrm{h}(x) \defeq f(x) - F(D^2 \phi, D \phi, \phi, x).
$$

Now, we define $\hat{f}_{u_{\varepsilon}}(x) \defeq \mathrm{h}^+(x) \Psi_{\varepsilon}(u_{\varepsilon} - \phi) + f(x) - \mathrm{h}^+(x)$. Then, note that, $\mathrm{h} \in L^p(\Omega)$ with the estimate
\begin{eqnarray}\label{5.4}
	\|\mathrm{h}\|_{L^p(\Omega)} &\le& \|f\|_{L^p(\Omega)} + \|F(D^2 \phi, D \phi, \phi, x) \|_{L^p(\Omega)}\nonumber \\
	&\le& \mathrm{C} \cdot \left( \|f\|_{L^p(\Omega)} + \|\phi\|_{W^{2,p}(\Omega)}\right)
\end{eqnarray}
for some $\mathrm{C}=\mathrm{C}(n,\lambda,\Lambda, \sigma, \xi, p )>0$, since $F$ fulfills (A1). In a consequence, we have
\begin{equation}\label{4.3}
  \|\hat{f}_{u_{\varepsilon}}\|_{L^p(\Omega)} \leq \mathrm{C}(n,\lambda,\Lambda, \sigma, \xi, p )\left( \|f\|_{L^p(\Omega)} + \|\phi\|_{W^{2,p}(\Omega)}\right)
\end{equation}

Next, we claim that the problem \eqref{NN4-4} admits a viscosity solution that enjoys \textit{a priori} estimates.  In effect, according to Perron's Method (see Lieberman's Book \cite[Theorem 7.19]{Leiberman}) it follows that for each $v_0 \in L^p(\Omega)$ fixed, there exists a unique viscosity solution $u_{\varepsilon} \in W^{2,p}(\Omega)$ satisfying in the viscosity sense
$$
		\left\{
		\begin{array}{rclcl}
			F(D^2u_{\varepsilon},Du_{\varepsilon},u_{\varepsilon},x) &=& \mathrm{h}^+(x) \Psi_{\varepsilon}(v_0 - \phi) + f(x) - \mathrm{h}^+(x)& \mbox{in} & \Omega \\
			\mathcal{B}(x,u_{\varepsilon},Du_{\varepsilon})&=& g(x)  & \mbox{on} &\partial \Omega,
		\end{array}
		\right.
$$
fulfilling (due to Theorem \ref{T1}) the estimate
\begin{equation}\label{Eq_Est_Obst}
  \|u_{\varepsilon}\|_{W^{2,p}(\Omega)}  \le  \mathrm{C} \cdot \left(\|u_{\varepsilon}\|_{L^{\infty}(\Omega)} + \|\hat{f}_{v_0}\|_{L^p(\Omega)}+\|g\|_{C^{1,\alpha}(\partial \Omega)}\right)
\end{equation}
for some $\mathrm{C}= \mathrm{C}(n,\lambda,\Lambda, p, \mu_0, \sigma, \omega, \|\beta\|_{C^2(\partial \Omega)}, \|\gamma\|_{C^2(\partial \Omega)}, \theta_0, \textrm{diam}(\Omega))>0$. Moreover, as in \eqref{4.3}, we have $\hat{f}_{v_0} \in L^p(\Omega)$.

Now, from A.B.P. Maximum Principle (Lemma \ref{ABP-fullversion}) we obtain
\begin{equation}\label{Eq_Est_Obst}
  \begin{array}{rcl}
  \|u_{\varepsilon}\|_{W^{2,p}(\Omega)} & \le & \mathrm{C} \cdot \left(\|\hat{f}_{v_0}\|_{L^p(\Omega)}+\|g\|_{C^{1,\alpha}(\partial \Omega)}\right)\\
&\le& \mathrm{C} \cdot\left( \|f\|_{L^p(\Omega)} + \|h\|_{L^p(\Omega)} + \|g\|_{C^{1,\alpha}(\partial \Omega)}\right)\\
&\le& \hat{\mathrm{C}}\cdot \left(\|f\|_{L^p(\Omega)} + \|g\|_{C^{1,\alpha}(\partial \Omega)} + \|\phi\|_{W^{2,p}(\Omega)} \right).
\end{array}
\end{equation}
Thus,
$$
	\|u_{\varepsilon}\|_{W^{2,p}(\Omega)} \le \mathrm{C}_0(\hat{\mathrm{C}}, n,\lambda,\Lambda,p,\mu_0, \sigma, \omega, \|\beta\|_{C^2(\partial \Omega)}, \theta_0, \textrm{diam}(\Omega), \|f\|_{L^p(\Omega)}, \|g\|_{C^{1,\alpha}(\partial \Omega)}, \|\phi\|_{W^{2,p}(\Omega)}),
$$
where $\mathrm{C}_0>0$ is independent on $v_0$.

At this point, by defining the operator $\mathcal{T}: L^p(\Omega) \rightarrow W^{2,p}(\Omega) \subset L^p(\Omega)$ given by $\mathcal{T}(v_0)=u_{\varepsilon}$, we conclude that $\mathcal{T}$ maps the $\mathrm{C}_0$-ball (in $L^p(\Omega)$) into itself. Hence, $\mathcal{T}$ is a compact operator. Therefore, by Schauder's fixed point theorem, there exists $u_{\varepsilon}$ such that $\mathcal{T}(u_{\varepsilon})=u_{\varepsilon}$, which is a viscosity solution to \eqref{NN4-4}.

Finally, we are in a position to establish our main result.

\begin{proof}[{\bf Proof of Theorem \ref{T3}}]

From \eqref{Eq_Est_Obst} we observe that
$$
	\|u_{\varepsilon}\|_{W^{2,p}(\Omega)} \le \mathrm{C}\cdot \left( \|f\|_{L^p(\Omega)} + \|g\|_{C^{1,\alpha}(\partial \Omega)} + \|\phi\|_{W^{2,p}(\Omega)}\right)
$$
for some $\mathrm{C}=\mathrm{C}(n,\lambda, \Lambda, p, \mu_0, \sigma, \xi, \|\beta\|_{C^2(\partial \Omega)}, \textrm{diam}(\Omega), \theta_0)$.  This ensures that $\{u_{\varepsilon}\}_{\varepsilon >0}$ is uniformly bounded in $W^{2,p}(\Omega)$. Hence, by standard compactness arguments, we can find a subsequence $\{u_{\varepsilon_j}\}_{j \in \mathbb{N}}$ with $\varepsilon_j \to 0$ and a function $u_{\infty} \in W^{2,p}(\Omega)$ such that
$$
\left\{
\begin{array}{rcl}
  u_{\varepsilon_j} \rightharpoonup u_{\infty} & \text{in} & W^{2, p}(\Omega)\\
 u_{\varepsilon_j} \to u_{\infty} & \text{in}& C^{0,\alpha_0}(\overline{\Omega})\\
 u_{\varepsilon_j} \to u_{\infty} & \text{in}& C^{1,1-\frac{n}{p}}(\overline{\Omega})
\end{array}
\right.
$$
for some universal $\alpha_0 \in (0,1)$.

Now, we assert that $u_{\infty}$ is a viscosity solution of \eqref{obss1}. In effect:

\begin{enumerate}
  \item[\checkmark] Since  $\beta(x) \cdot Du_{\varepsilon_j}(x) + \gamma(x) u_{\varepsilon_j}(x) =g(x)$ and $\{u_{\varepsilon_j}\}_{j \in \mathbb{N}}$ are uniformly bounded and equi-continuous on $\partial \Omega$, then from \eqref{Eq_Est_Obst} and Morrey embedding $W^{2, p}(\Omega) \hookrightarrow C^{1, 1-\frac{n}{p}}$, we have
$$
	\beta(x) \cdot Du_{\infty}(x) + \gamma(x) u_{\infty}(x)= g(x) \quad \textrm{on} \quad \partial \Omega
$$
in the viscosity sense (and point-wisely).
  \item[\checkmark]  On the other hand, from \eqref{NN4-4}, we have
\begin{eqnarray*}
	F(D^2 u_{\varepsilon_j}, Du_{\varepsilon_j}, u_{\varepsilon_j},x) &=& \mathrm{h}^+(x) \Psi_{\varepsilon_j}(u_{\varepsilon_j} - \phi) + f(x) - \mathrm{h}^+(x)\\
	& \le& f(x) \quad \textrm{in} \quad \Omega \quad \text{for each} \quad  j \in \mathbb{N}.
\end{eqnarray*}
Thus, by applying Stability results (Lemma \ref{Est}), by passing to the limit $j \to +\infty$, we can conclude
$$
	F(D^2 u_{\infty}, Du_{\infty}, u_{\infty}, x) \le f \quad \textrm{in} \quad \Omega.
$$

  \item[\checkmark] Next, we are going to prove that
$$
	u_{\infty} \ge \phi \quad \textrm{in} \quad \overline{\Omega}.
$$
Firstly, we see that $\Psi_{\varepsilon_j}(u_{\varepsilon_j} - \phi) \equiv 0$ on the set $\mathcal{O}_j \defeq \{x \in \overline{\Omega} \, : \, u_{\varepsilon_j}(x) < \phi(x)\}$.  Note that, if $\mathcal{O}_j = \emptyset$, we have nothing to prove. Thus, we suppose that $\mathcal{O}_j \not= \emptyset$. Then,
$$
	F(D^2 u_{\varepsilon_j}, Du_{\varepsilon_j}, u_{\varepsilon_j},x) = f(x) - \mathrm{h}^+(x) \quad \textrm{for} \,\,\, x \in \mathcal{O}_j.
$$
Now, for each $j \in \mathbb{N}$  note that $\mathcal{O}_j$ is relatively open in $\overline{\Omega}$, since $u_{\varepsilon_j} \in C^0(\overline{\Omega})$. Furthermore, from definition of $\mathrm{h}$ we have
$$
F(D^2 \phi, D \phi, \phi, x ) = f(x)-\mathrm{h}(x) \ge F(D^2 u_{\varepsilon_j}, Du_{\varepsilon_j}, u_{\varepsilon_j}, x) \,\,\, \textrm{in} \,\,\, \mathcal{O}_j
$$
Moreover, we have $u_{\varepsilon_j} = \phi$ on $\partial \mathcal{O}_j \setminus \partial \Omega$. Thus, from the Comparison Principle (see \cite[Theorem 2.10]{CCKS} and \cite[Theorem 7.17]{Leiberman}) we conclude that $u_{\varepsilon_j} \ge \phi$ in $\mathcal{O}_j$,  which yields a contradiction. Therefore, $\mathcal{O}_j = \emptyset$ and $u_{\infty} \ge \phi$ in $\overline{\Omega}$.

  \item[\checkmark] Finally, it remains to show that
$$
	F(D^2 u_{\infty}, Du_{\infty}, u_{\infty} ,x) = f(x) \,\,\,\, \textrm{in} \,\,\,\, \{ u_{\infty} > \phi \}
$$
in the viscosity sense. For this end, for each $ k \in \mathbb{N}$ we notice that
$$
	\mathrm{h}^+(x)\Psi_{\varepsilon_j}(u_{\varepsilon_j} -\phi) + f(x)-\mathrm{h}^{+}(x) \to f(x) \,\,\, \textrm{a.e. in} \,\,\, \left\{x \in \Omega \, : \, u_{\infty}(x) > \phi(x) + \frac{1}{k} \right\} .
$$
Therefore, via Stability results (Lemma \ref{Est}) we conclude (in the viscosity sense) that
$$
	F(D^2 u_{\infty}, Du_{\infty}, u_{\infty} ,x) = f(x) \,\,\, \textrm{in} \,\,\, \{u_{\infty} > \phi\} = \bigcup_{k=1}^{\infty} \left\{ u_{\infty} > \phi + \frac{1}{k}\right\} \quad \text{as} \quad  j \to +\infty,
$$
thereby proving the claim.
\end{enumerate}

Finally, from \eqref{Eq_Est_Obst} $u_{\infty}$ satisfies the following regularity estimate
{\scriptsize{
\begin{eqnarray*}
	\|u_{\infty}\|_{W^{2,p}(\Omega)} &\le& \liminf_{j \to +\infty} \|u_{\varepsilon_j}\|_{W^{2,p}(\Omega)} \\ &\le& \mathrm{C}(n,\lambda, \Lambda, p, \mu_0, \sigma, \omega, \|\beta\|_{C^2(\partial \Omega)}, \|\gamma\|_{C^2(\partial \Omega)}, \textrm{diam}(\Omega), \theta_0)\cdot\left( \|f\|_{L^p(\Omega)} + \|g\|_{C^{1,\alpha}(\partial \Omega)} + \| \phi \|_{W^{2,p}(\Omega)}\right)
\end{eqnarray*}}}
thereby finishing the proof of the Theorem.
\end{proof}

As a consequence, we obtain the uniqueness of existing solutions.

\begin{corollary}[\textbf{Uniqueness}] The viscosity solution found in the Theorem \ref{T3} is unique.
\end{corollary}
\begin{proof}
Indeed, let $u$ and $v$ be two viscosity solutions of \eqref{obss1}. Assume that $u \not= v$. Then, we may suppose without loss of generality that
$$
	\mathcal{O}_{\sharp} = \{v > u\} \not= \emptyset.
$$
Since $v > u \ge \phi$ in $\mathcal{O}_{\sharp}$, we obtain in the viscosity sense
$$
	F(D^2 v, Dv, v,x) = f(x) \quad \textrm{in} \quad \mathcal{O}_{\sharp}
$$
Hence, we conclude that
$$
		\left\{
		\begin{array}{rclclcc}
			F(D^2u,Du,u,x) &\le& f(x) & \le & F(D^2 v, Dv,v,x)& \mbox{in} & \mathcal{O}_{\sharp} \\
			& & u(x)&=& v(x)  & \mbox{on} &\partial \mathcal{O}_{\sharp} \setminus \partial \Omega,\\
			\mathcal{B}(x,u,Du) &=& g(x) & =& \mathcal{B}(x,v,Dv) & \mbox{on} & \partial \mathcal{O}_{\sharp} \cap \partial \Omega
		\end{array}
		\right.
$$

Therefore, according to Comparison Principle for problems with oblique boundary conditions \cite[Theorem 2.10]{CCKS} and \cite[Theorem 7.17]{Leiberman}, we conclude that $u \ge v$ in $\mathcal{O}_{\sharp}$ whether $\partial \mathcal{O}_{\sharp} \cap \partial \Omega = \emptyset$ or not. However, this contradicts the definition of the set $\mathcal{O}_{\sharp}$, thereby proving that $u =v$.
\end{proof}

\section{Final conclusions: Density of viscosity solutions}\label{Sec_Density}

\hspace{0.4cm}In this final part, we will deliver another application to $W^{2,p}$ regularity estimates.

\begin{theorem}[{\bf $W^{2,p}$ density on the class $C^{0}-$viscosity solutions}]
Let $u$ be a $C^{0}-$viscosity solution of
$$
\left\{
\begin{array}{rclcl}
  F(D^{2}u,x) & = & f(x) & \text{in} & \mathrm{B}^{+}_{1} \\
  \mathcal{B}(x, u, Du) & = & g(x) & \text{on} & \mathrm{T}_{1},
\end{array}
\right.
$$
where $f\in L^{p}(\mathrm{B}^{+}_{1})\cap C^{0}(\mathrm{B}^{+}_{1})$ (for $n \le p<\infty$),  $\beta,\gamma, g\in C^{1,\alpha}(\mathrm{T}_{1})$ with $\gamma\leq 0$ and $\beta \cdot \overrightarrow{\textbf{n}} \geq \mu_{0}$ on $\mathrm{T}_{1}$, for some $\mu_{0}>0$. Then, for any $\delta>0$, there exists a sequence $(u_{j})_{j\in\mathbb{N}}\subset W^{2,p}_{loc}(\mathrm{B}^{+}_{1})\cap \mathcal{S}(\lambda-\delta,\Lambda+\delta,f)$ converging local uniformly to $u$.
\end{theorem}

\begin{proof}
The proof follows some ideas from \cite[Theorem 8.1]{PT} adapted to oblique boundary scenery. We will present the details to the reader for completeness. Firstly, we are going to build up such a desired sequence of operators $F_{j}:Sym(n)\times \mathrm{B}^{+}_{1}\longrightarrow\mathbb{R}$ as follow: Given $\delta>0$ we consider the Pucci Maximal operator \begin{eqnarray*}
\mathscr{L}_{\delta}(\mathrm{X})\defeq \mathscr{P}^{+}_{(\lambda-\delta),(\Lambda+\delta)}(\mathrm{X})=(\Lambda +\delta)\sum_{e_i >0} e_i(\mathrm{X}) +(\lambda-\delta) \sum_{e_i <0} e_i(\mathrm{X}),
\end{eqnarray*}
where $e_{i}(\mathrm{X})$ are the eigenvalues of matrix $\mathrm{X}\in \text{Sym}(n)$. Now, we define
\begin{eqnarray*}
F_{j}:&\text{Sym}(n)\times \mathrm{B}^{+}_{1}&\longrightarrow \mathbb{R}\\
&(\mathrm{X},x)&\longmapsto \max\{F(\mathrm{X},x),\mathscr{L}_{\delta}(\mathrm{X})-\mathrm{C}_{j}\},
\end{eqnarray*}
for $(\mathrm{C}_{j})_{j\in\mathbb{N}}$ a sequence of divergence positive number (chosen precisely \textit{a posteriori}). Thus, notice that $F_{j}$ is continuous (because it is the maximum between two continuous functions) and uniformly elliptic with ellipticity constants $\lambda-\delta\leq \Lambda+\delta$, as $F$ and $\mathscr{L}_{\delta}$ are. Now, taking into account  the $(\lambda,\Lambda)$-ellipticity of $F$ there holds
\begin{eqnarray*}
F(\mathrm{X},x)&\geq& \lambda\sum_{e_i >0} e_i(\mathrm{X}) +\Lambda\sum_{e_i <0} e_i\\
&\geq& \lambda\displaystyle\sum_{e_i >0}e_{i}(\mathrm{X})-\Lambda\| \mathrm{X}\|\\
&=& \mathscr{L}_{\delta}(\mathrm{X})-(\Lambda+\delta-\lambda)\sum_{e_i >0}e_{i}-(\lambda-\delta)\sum_{e_i >0}e_{i}-\Lambda\| \mathrm{X}\|\\
&\geq& \mathscr{L}_{\delta}(\mathrm{X})-(2\Lambda -\lambda+\delta)\| \mathrm{X}\|\\
&\geq& \mathscr{L}_{\delta}(\mathrm{X})-\mathrm{C}_{j}, \ \forall \mathrm{B}_{j}\subset \text{Sym}(n), \ \forall x\in \mathrm{B}^{+}_{1},
\end{eqnarray*}
where we select $\mathrm{C}_{j}\defeq j(2\Lambda-\lambda+\delta)$. Hence, $F\equiv F_{j}$ in $\mathrm{B}_{j}\times \mathrm{B}^{+}_{1}\subset Sym(n)\times \mathrm{B}^{+}_{1}$. On the other hand, we investigate the recession operator associated to $F_{j}$. For that purpose, for each $\mu>0$ we have
\begin{eqnarray*}
(F_{j})_{\mu}(\mathrm{X},x)=\mu F_{j}(\mu^{-1}\mathrm{X},x)=\max\{F_{\mu}(\mathrm{X},x),L_{\delta}(\mathrm{X})-\mu C_{j}\}.
\end{eqnarray*}

Since $F_{\mu}$ is $(\lambda,\Lambda)$-elliptic for any $\mu>0$ we get
\begin{eqnarray*}
F_{\mu}(\mathrm{X},x)&\leq& \Lambda\sum_{e_i >0} e_i(\mathrm{X}) +\lambda\sum_{e_i <0} e_i(\mathrm{X})\\
& = & \mathscr{L}_{\delta}(\mathrm{X})-\delta \sum_{e_i >0} e_{i}(\mathrm{X})+\delta\sum_{e_i <0} e_{i}(\mathrm{X})\\
&\leq& \mathscr{L}_{\delta}(\mathrm{X})-\delta\|\mathrm{X}\|\\
&\leq & \mathscr{L}_{\delta}(\mathrm{X})-\mu \mathrm{C}_{j},
\end{eqnarray*}
for all $ \mathrm{X}\in \mathrm{B}_{\frac{\mu \mathrm{C}_{j}}{\delta}}\subset \text{Sym}(n)$ and $x\in \mathrm{B}^{+}_{1}$. Therefore, we conclude that $F_{j}(\mathrm{X},x)= \mathscr{L}_{\delta}(\mathrm{X})-\mathrm{C}_{j}$ outside a ball of radius $\sim \mathrm{C}_{j}$, then $F_{j}^{\sharp}=\mathscr{L}_{\delta}$. Now, notice that $F_{j}^{\sharp}$ fulfills $C^{1,1}$ \textit{a priori} estimates, i.e., given $g_{0}\in C^{1,\alpha}(\mathrm{T}_{1})$ any viscosity solution to
$$
\left\{
\begin{array}{rclcl}
  F_{j}^{\sharp}(D^{2}\mathfrak{h}, x_{0}) & = & 0 & \mbox{in} & \mathrm{B}^{+}_{1}\\
  \mathcal{B}(x, \mathfrak{h}, D\mathfrak{h}) & = & g_{0}(x) & \mbox{on}  & \mathrm{T}_{1}
\end{array}
\right.
$$
satisfies $\mathfrak{h}\in C^{1,1}(\overline{\mathrm{B}^{+}_{1/2}})$ (see e.g., \cite[Teorema 1.3]{LiZhang}) with the following estimate:
\begin{eqnarray*}
\|\mathfrak{h}\|_{C^{1,1}(\overline{\mathrm{B}^{+}_{1/2}})}\leq \mathrm{C}\cdot \left(\|\mathfrak{h}\|_{L^{\infty}(\mathrm{B}^{+}_{1})}+\|g\|_{C^{1,\alpha}(\overline{\mathrm{T}_{1}})}\right).
\end{eqnarray*}

Therefore, statements of Proposition \ref{T-flat} are in force.
Then, for each fixed $j\in\mathbb{N}$, any viscosity solution of
$$
\left\{
\begin{array}{rcrcl}
  F_{j}(D^{2}v, x) & = & f(x) & \text{in} & \mathrm{B}^{+}_{1} \\
  \mathcal{B}(x, v, Dv) & = & g(x) & \text{on} & \mathrm{T}_{1}
\end{array}
\right.
$$
enjoys $W^{2,p}$ regularity estimates. To be more specific, for each $j\in\mathbb{N}$ there exists a constant $\kappa_{j}>0$ such that
\begin{eqnarray*}
\|v\|_{W^{2,p}(\mathrm{B}_{1/2}^{+})}\leq \kappa_{j}\cdot (\|v\|_{L^{\infty}(\mathrm{B}_{1}^{+})}+\|f\|_{L^{p}(\mathrm{B}^{+}_{1})}+\|g\|_{C^{1,\alpha}(\overline{\mathrm{T}_{1}})}).
\end{eqnarray*}

Finally, we build up the desired sequence $(u_{j})_{j \in \mathbb{N}}$ to be a viscosity solution to
$$\left\{
\begin{array}{rcrcl}
  F_{j}(D^{2}u_{j},x) & = & f(x) & \mbox{in} & \mathrm{B}^{+}_{1} \\
\mathcal{B}(x, u_{j}, Du_{j}) & = & g(x) & \mbox{on} & \mathrm{T}_{1}\\
u(x) & = & u_{j}(x) & \mbox{on} & \partial \mathrm{B}^{+}_{1}\setminus \mathrm{T}_{1},
\end{array}
\right.
$$
whose the existence does hold from Theorem \ref{Existencia}. Furthermore, each $u_{j}\in W^{2,p}(\mathrm{B}^{+}_{1})$ and since operators $F_{j}$ are uniformly elliptic, it follow that $F_{j}(\cdot,x)\to F_{0}(\cdot,x)$ local uniformly in $\text{Sym}(n)$ for each $x\in \mathrm{B}^{+}_{1}$. Furthermore, since $F_{j}\equiv F$ in $\mathrm{B}_{j}\times \mathrm{B}^{+}_{1}$ then $F_{0}\equiv F$. In conclusion, using local $C^{0, \alpha}$ regularity and A.B.P. estimates (see Theorem \ref{Holder_Est} and Lemma \ref{ABP-fullversion} respectively) the sequence $(u_{j})_{j \in \mathbb{N}}$, converge, up to a subsequence, local uniformly to $u_{0}$ in the $C^{0, \alpha}$-topology. In addition, according to Stability results (Lemma \ref{Est}) $u_{0}$ is a viscosity solution of
$$
\left\{
\begin{array}{rclcl}
  F(D^{2}u_{0},x) & = & f(x) & \mbox{in} & \mathrm{B}^{+}_{1}\\
  \mathcal{B}(x, u_{0}, Du_{0}) & = & g(x) & \mbox{on} & \mathrm{T}_{1}\\
  u(x) & = & u_0(x) & \mbox{on} & \partial \mathrm{B}^{+}_{1}\setminus \mathrm{T}_{1}.
\end{array}
\right.
$$

Finally, by taking $w=u_{0}-u$, we can see that $w$ satisfies in the viscosity sense
$$
\left\{
\begin{array}{rcl}
  w\in \mathcal{S}(\lambda/n,\Lambda,0)  &  \mbox{in} & \mathrm{B}^{+}_{1}\\
  \mathcal{B}(x, w, Dw)=0  & \mbox{on} & \mathrm{T}_{1}\\
  w=0 & \mbox{on} & \partial \mathrm{B}^{+}_{1}\setminus \mathrm{T}_{1},
\end{array}
\right.
$$
and, once again, by A.B.P. estimates (Lemma \ref{ABP-fullversion}) we conclude that $w=0$ in $\overline{\mathrm{B}^{+}_{1}}\setminus \mathrm{T}_{1}$. Therefore, $w\equiv 0$, i.e., $u=u_{0}$, thereby finishing the proof.

\end{proof}

\subsection*{Acknowledgments}

\hspace{0.4cm} J.S. Bessa was partially supported by CAPES-Brazil under Grant No. 88887.482068/2020-00. J.V. da Silva, M. N. Barreto Frederico and G.C. Ricarte have been partially supported by CNPq-Brazil under Grant No. 307131/2022-0, No 165746/2020-3, and No. 304239/2021-6.  J.V. da Silva has been partially supported by   FAPDF Demanda Espont\^{a}nea 2021 and   FAPDF - Edital 09/2022 - DEMANDA ESPONTÂNEA. Part of this work was developed during the \textit{Fortaleza Conference on Analysis and PDEs} (2022) at the Universidade Federal do Cear\'{a} (UFC-Brazil). J.V. da Silva would like to thank to UFC's Department of Mathematics for fostering a pleasant scientific and research atmosphere during his visit in the Summer of 2022.

\end{document}